\documentclass[english]{article} %
\usepackage[T1]{fontenc} %
\usepackage[utf8]{inputenc} %
\usepackage{color} %
\makeatletter %
\usepackage{amsthm} %
\usepackage{amsmath} %
\allowdisplaybreaks %
\usepackage{amssymb} %
\usepackage{esint} %
\usepackage[authoryear]{natbib} %
\usepackage{nameref} %
\usepackage{graphics} %
\usepackage{graphicx} %
\usepackage{bm} %
\usepackage{subfig} %
\usepackage{booktabs} %
\usepackage[affil-it]{authblk} %
\usepackage{lmodern} %
\usepackage{setspace}

\makeatother \usepackage{babel}

\newcommand{\E}{\mathbb{E}} %
\newcommand{\N}{\mathbb{N}}

\renewcommand{\L}{\mathbb{L}}
\renewcommand{\P}{\mathbb{P}} %
\newcommand{\R}{\mathbb{R}} %
\newcommand{\V}{\operatorname{Var}}
\newcommand{\supp}{\operatorname{supp}}
\newcommand{\C}{\operatorname{Cov}}
\newcommand{\B}{\operatorname{Bias}}

\newcommand{\vech}{\operatorname{vech}}

\theoremstyle{plain} \newtheorem{thm}{\protect\theoremname}
\theoremstyle{plain} \newtheorem{prop}{\protect\propositionname}
\theoremstyle{plain} \newtheorem{assumption}{\protect\assumptionname}
\theoremstyle{plain} \newtheorem{cor}{\protect\corollaryname}
\theoremstyle{plain} \newtheorem{lem}{\protect\lemmaname}
\theoremstyle{remark} 
\theoremstyle{definition} 
\theoremstyle{definition} 
\theoremstyle{definition} 

\providecommand{\assumptionname}{Assumption}
\providecommand{\lemmaname}{Lemma}
\providecommand{\propositionname}{Proposition}
\providecommand{\remarkname}{Remark}
\providecommand{\corollaryname}{Corollary}
\providecommand{\theoremname}{Theorem}
\providecommand{\definitionname}{Definition}
\providecommand{\notationname}{Notation}
\providecommand{\conditionname}{Condition}

\begin{document}

\date{}

\title{Efficient estimation of conditional covariance matrices for
  dimension reduction}

\author[1]{S\'{e}bastien Da Veiga} %
\author[2]{Jean-Michel Loubes} %
\author[2]{Maikol Sol\'{i}s}

\affil[1]{Institut Fran\c{c}ais du P\'{e}trole, Paris, France.}%
\affil[2]{Institut de Math\'{e}matiques de Toulouse, Universit\'{e} Paul
  Sabatier, Toulouse, France.}

\renewcommand\Authands{ and }
\maketitle

\begin{abstract}
  Let $\boldsymbol{X}\in \R^p$ and $Y\in \R$. %
  In this paper we propose an estimator of the conditional covariance
  matrix, $\C(\E[\boldsymbol{X}\vert Y])$, in an inverse regression
  setting. %
  Based on the estimation of a quadratic functional, this methodology
  provides an efficient estimator from a semi parametric point of
  view. %
  We consider a functional Taylor expansion of
  $\C(\E[\boldsymbol{X}\vert Y])$ under some mild conditions and the
  effect of using an estimate of the unknown joint distribution.
  The asymptotic properties of this estimator are also provided.
\end{abstract}

\section{Introduction}

\footnotetext[1]{Institut Fran\c{c}ais du P\'{e}trole, Paris,
  France.}%
\footnotetext[2]{Institut de Math\'{e}matiques de Toulouse, Universit
  Paul Sabatier, Toulouse, France.}

Consider the nonparametric regression

\begin{equation*}
  Y = \varphi (\bm{X}) + \epsilon
\end{equation*}

where $\bm{X}\in\R^{p}$, $Y\in\R$, $\varphi$ is a unknown function
from $\R^p$ to $\R,$ and $\epsilon$ is random noise with
$\E[\epsilon]=0$. %

Assume we observe an independent identically distributed sample,
$\bm{X}_k=(X_{1k}, \ldots, X_{pk})$ and $Y_k$ for $k=1, \ldots, n$. %
If we face a model with more variables than observed data (i.e., $p\gg
n$), the high-dimensional setting blurs the relationship between
$\bm{X}$ and $Y$. %
The literature calls this phenomenon the \emph{curse of
  dimensionality}. %


Many methods have been developed to overcome this issue. %
In particular, in ~\cite{li1991sliced} the sliced inverse regression
method is proposed. %

The authors considered the following model

\begin{equation}
  \label{eq:taylor:semipara_model}
  Y = \phi(\upsilon_{1}^{\top}\bm{X}, \ldots,
  \upsilon_{K}^{\top}\bm{X}, \epsilon)
\end{equation}

where the $\upsilon$'s are unknown vectors in $\R^{p}$, the
$\epsilon_k$'s are independent of $X_k$, and $\phi$ is an arbitrary
function in $\R^{K+1}$. %

This model gathers all the relevant information about the variable,
$Y$, with only a projection onto the $K\ll p$ dimensional subspace,
$(\upsilon_{1}^{\top}\bm{X}, \ldots, \upsilon_{K}^{\top}\bm{X})$. %
If $K$ is small, this method reduces the dimension by estimating the
$\upsilon$'s efficiently. %

We call the $\upsilon$'s effective dimension reduction directions. %
This method is a semi parametric method since the unknown density
blurs the estimation of the parameters $\upsilon$'s which are the
projection directions.

For a review on sliced inverse regression methods, we refer to
\citet{li1991sliced, li1991rejoinder}, \cite{duan1991slicing}
\cite{hardle1991slicedcomment}, and references therein. %
In short, the eigenvectors associated with the largest eigenvalues of
$\C(\E[\bm{X}\vert Y])$ are the model \eqref{eq:taylor:semipara_model}
effective dimension reduction directions. %
Therefore, if we better estimate the conditional covariance,
$\C(\E[\bm{X}\vert Y])$, then we will better approximate the reduction
dimension space. %
This conditional covariance can be written as

\begin{equation*}
  \Sigma = \C(\E[\boldsymbol{X}\vert Y]) = \E[\E[\boldsymbol{X}\vert
  Y]\E[\boldsymbol{X}\vert Y]^{\top}] -
  \E[\boldsymbol{X}]\E[\boldsymbol{X}]^{\top}.
\end{equation*}

where $A^{\top}$ denotes the transpose of $A$. %
Since $\E[\bm{X}]\E[\bm{X}]^{\top}$ can be estimated easily, we point
out that the estimation of the matrix, $\E[\E[\bm{X} \vert Y]
\E[\bm{X}\vert Y]^{\top}]$, is the part that we should focus on.

Therefore, we refer for instance to \citet{zhu1996asymptotics} and
\cite{ferre2003functional,ferre2005smoothed} who used kernel
estimators. Also, we refer to ~\citet{hsing1999nearest} who combined
the nearest neighbor and the sliced inverse regression and
to~\citet{bura2001estimating} who assumed a parametric form for the
conditional vector, $\E[\bm{X}\vert Y]$. Lastly, we refer to
~\citet{setodji2004kmeans} who used a k-means method and to
~\citet{cook2005sufficient} who rewrote the sliced inverse regression
as a least square minimization.

The estimation of each element of $\Sigma = (\sigma_{ij})$ for $i,j=1,
\ldots, p$ depends on the joint distribution of $(\bm{X},Y)$.
Therefore, we propose a plug-in estimator for each parameter,
$\sigma_{ij}$. %
The quadratic functional estimator of $\E[\E[\boldsymbol{X}\vert
Y]\E[\boldsymbol{X}\vert Y]^{\top}]$ is obtained using a Taylor
expansion based on the ideas of \cite{daveiga2013efficient}. %
The first order term drives the asymptotic convergence while the
higher order terms will be shown to be negligible. %
These ideas are driven by previous studies from
\citet{laurent1996efficient} on the estimation of quadratic
integrals. %
Finally, we obtain a semi parametric estimate which is shown to be
efficient. %
This estimator offers an alternative method for plug-in methods for
conditional covariance matrices with minimum variance properties.

The organization of this paper is as follows. %
Section \ref{sec:taylor:Methodology} motivates our investigation of
$\C(\E[\boldsymbol{X}\vert Y])$ using a Taylor approximation. %
In Section \ref{sub:taylor:Hypothesis-and-Assumptions} the notations
and hypotheses are set up. %
We demonstrate, in Section \ref{sub:taylor:Efficient-Estimation-of},
the efficient convergence of each coordinate for our estimator. %
We also state the asymptotic normality for the whole matrix. %
For the quadratic term of the Taylor expansion of
$\C(\E[\boldsymbol{X}\vert Y])$, we find an asymptotic bound for the
variance in Section \ref{sec:taylor:Estimation-of-quadratic}. %
All of the technical Lemmas and related proofs will be found in
Sections \ref{sec:taylor:Technical-Results} and
\ref{sec:taylor:Proofs}, respectively. %

\section{Methodology}
\label{sec:taylor:Methodology}

Let $\boldsymbol{X}\in\mathbb{R}^{p}$ be a squared integrable random
vector with $p\geq1$ and let $Y\in\mathbb{R}$ be a random variable. %
We will denote $X_i$ and $X_j$ as the $i$-th and $j$-th coordinates of
$\boldsymbol{X}$, respectively. %
We denote by $f_{ij}(x_i, x_j, y)$ the joint density of the vector,
$(X_i, X_j, Y)$ for $i, j= 1\ldots p$. %
Recall that the density function, $f_{ij}$, depends on the indices,
$i$ and $j$, namely for each triplet $(X_i, X_j, Y)$ there exists a
joint density function called $f_{ij}(x_i, x_j, y)$. %
For the sake of simplicity, we will denote $f_{ij}$ only by $f$ to
avoid cumbersome notation. %
When $i$ is equal to $j$, we will call the joint density of $(X_i, Y)$
by $f_i(x_i, y)$, for $i=1, \ldots, p$. %
When the context is clear, we can name $f_i$ simply by $f$. %
Finally, let $f_Y(\cdot)= \int_\R f(x_{i}, x_{j}, \cdot)\, dx_{i}\,
dx_{j}$ be the marginal density function with respect to $Y$. %
Also, without loss of generality, we assume that the variables,
$\bm{X}$ are centered, i.e., $\E[\bm{X}] = 0$. %

Given a sample of $(\boldsymbol{X}, Y)$, we aimed to study the
asymptotic and efficiency properties of $\sigma_{ij}$. %
Then we will obtain similar results for the whole matrix $\Sigma$. %

The work of \cite{laurent1996efficient} and
\cite{daveiga2013efficient} already deals with the estimation of the
diagonal of $\C(\E[\boldsymbol{X}\vert Y])$. %
In this work, we shall extend their methodologies to the case $i \neq
j $ in order to find an alternative estimator for the sliced inverse
regression directions.

Recall that
\begin{align*}
  \Sigma = \C(\E[\boldsymbol{X}\vert Y]) & =\E[\E[\boldsymbol{X}\vert
  Y]\E[\boldsymbol{X}\vert Y]^{\top}].
\end{align*}

We then defined each entry of the conditional covariance matrix,
$\Sigma$ as %
\begin{equation*}
  \sigma_{ij} = \E[\E[X_i\vert Y]\E[X_j\vert Y]^{\top}]
  \quad i, j=1, \ldots, p.
\end{equation*}

Notice that we can write each $\sigma_{ij}$ for $i\neq j$ as
\begin{multline}
  \label{eq:taylor:sigma_ij}
  T_{ij}(f) = \sigma_{ij} = \int \left(\frac{\int x_{i}f(x_{i}, x_{j},
      y)\, dx_{i}\, dx_{j}}
    {f_Y(y)}\right) \\
  \left(\frac{\int x_{j}f(x_{i}, x_{j}, y)\, dx_{i}\, dx_{j}}
    {f_Y(y)}\right) f(x_{i}, x_{j}, y)\, dx_{i}\, dx_{j}\, dy.
\end{multline}
where
\begin{multline}
  \label{eq:taylor:Def_Tijf}
  T_{ij}(\psi) = \int\left(\frac{\int x_{i}\psi(x_{i}, x_{j}, y)\,
      dx_{i}\, dx_{j}}
    {\int \psi(x_{i}, x_{j}, y)\, dx_{i}\, dx_{j}}\right) \\
  \left(\frac{\int x_{j}\psi(x_{i}, x_{j}, y)\, dx_{i}\, dx_{j}} {\int
      \psi(x_{i}, x_{j}, y)\, dx_{i}\, dx_{j}}\right) \psi(x_{i},
  x_{j}, y) \, dx_{i}\, dx_{j}\, dy. %
\end{multline}
and $\psi$ is a square, integrable function in $\L^2(dx_i\, dx_j,
dy)$.

The functional, $T_{ij}(\psi)$, is defined for any square integrable,
$\psi$. %
If we take specifically $\psi=f$, then the functional, $T_{ij}(f)$, is
equal to the parameter $\sigma_{ij}$ defined in equation
\ref{eq:taylor:sigma_ij}. %

In order to obtain an estimator for $\sigma_{ij}$, we used a non
parametric estimator of $f$ and computed the estimation error. %
Suppose that we observe $(X_{ik}, X_{jk}, Y_k), \ k=1, \ldots, n$ as
an independent and identically distributed sample of $(X_i, X_j,
Y)$. %
To get rid of any dependency issues, we split this sample into two
independent subsamples of size $n_1$ and $n_2=n=n_1$. %
The first sample is used to build $\hat{f}$ as a preliminary estimator
of $f$. %
The other subsample will be used to estimate the parameters of the
conditional covariance matrix, $\sigma_{ij}$. %
The main idea is to expand $T_{ij}(f)$ in a Taylor series around a
neighborhood of $\hat{f}$. %

More precisely, we defined an auxiliar function, $F:[0,
1]\rightarrow\R$;
\begin{equation*}
  F(u) = T_{ij}(uf+(1-u)\hat{f})
\end{equation*}

with $u\in[0, 1]$. %

The Taylor expansion of $F$ between $0$ and $1$ up to the third order
is

\begin{equation}
  \label{eq:taylor:F-Taylor}
  F(1) = F(0)
  + F^{\prime}(0)
  + \frac{1}{2}F^{\prime\prime}(0)
  + \frac{1}{6}F^{\prime\prime\prime}(\xi)(1-\xi)^{3}
\end{equation}

for some $\xi\in[0, 1]$. %

Moreover, we have
\begin{align*}
  F(1) & =T_{ij}(f)\\
  F(0) & =T_{ij}(\hat{f})
\end{align*}

To simplify the notation set
\begin{equation*}
  m_{i}(f_{u}, y) =
  \frac{\int x_{i} f_{u}(x_{i}, x_{j}, y)\, dx_{i}\, dx_{j}}
  {\int f_{u}(x_{i}, x_{j}, y)\, dx_{i}\, dx_{j}},
\end{equation*}
where $f_{u}=uf+(1-u)\hat{f}$, for all $u$ belonging to $[0, 1]$. %
Notice that if $u=0$ then $m_{i}(f_{0}, y)=m_{i}(\hat{f}, y)$.%

We can rewrite $F(u)$ as
\begin{equation*}
  F(u) = \int m_{i} (f_{u}, y) m_{j} (f_{u}, y) f_{u}(x_{i}, x_{j}, y)
  \, dx_{i}\, dx_{j}\, dy. %
\end{equation*}

The next Proposition provides the $T_{ij}(f)$ of the Taylor
expansion. %

\begin{prop}[Linearization of the operator $T$]
  \label{prop:taylor:Decomposition_Taylor_Tijf} For the functional
  $T_{ij}(f)$ defined in Equation \eqref{eq:taylor:Def_Tijf}, the
  following decomposition holds
  \begin{multline}
    \label{eq:taylor:T-ij-Taylor}
    T_{ij}(f) = \int H_{1}(\hat{f}, x_{i}, x_{j}, y) f(x_{i}, x_{j},
    y)\,  dx_{i}\, dx_{j}\, dy \\
    + \int H_{2}(\hat{f}, x_{i1}, x_{j2}, y) f(x_{i1}, x_{j1}, y)
    f(x_{i2}, x_{j2}, y)\, dx_{i1}\, dx_{j1}\, dx_{i2}\ dx_{j2}\, dy
    \\
    + \Gamma_{n}
  \end{multline}
  where
  \begin{align}
    \label{eq:taylor:Gamma}
    H_{1}(\hat{f}, x_{i}, x_{j}, y) & = x_{i} m_{j}(\hat{f}, y) +
    x_{j}m_{i}(\hat{f}, y) - m_{i}(\hat{f}, y) m_{j}(\hat{f},
    y)\nonumber \\ 
    H_{2}(\hat{f}, x_{i1}, x_{j2}, y) & = \frac{1}{\int\hat{f}(x_{i},
      x_{j}, y)\, dx_{i}\, dx_{j}} (x_{i1} - m_{i}(\hat{f}, y))
    (x_{j2} - m_{j}(\hat{f}, y))\nonumber\\ 
    \Gamma_{n} & =\frac{1}{6} F^{\prime\prime\prime}(\xi) (1-\xi)^{3},
  \end{align}
  for some $\xi\in]0, 1[$. %
\end{prop}

This decomposition splits $T_{ij}(f)$ into three parts. %
A linear functional of $f$, which is easily estimated, a quadratic
functional, and an error term, $\Gamma_n$. %
The main part of this work was to control the quadratic term of
Equation \eqref{eq:taylor:T-ij-Taylor}. %
Since we used $n_{1}<n$ to build a preliminary approximation,
$\hat{f}$, we then used a sample of size $n_{2}=n-n_{1}$ to estimate
$\sigma_{ij}$. %
Note that the first term in Equation \eqref{eq:taylor:T-ij-Taylor} is
a linear functional in $f$, therefore its empirical estimator is
\begin{equation}
  \label{eq:taylor:estim-linear}
  \frac{1}{n_{2}}
  \sum_{k=1}^{n_{2}}
  H_{1}(\hat{f}, X_{ik}, X_{jk}, Y_{k}).
\end{equation}

Conversely, the second term is a nonlinear functional of $f$. %
It is a particular part of the issue of the estimation of general
quadractic functionals.

\begin{equation*}
  \theta(f) = \int \eta(x_{i1}, x_{j2}, y) f(x_{i1}, x_{j1}, y)
  f(x_{i2}, x_{j2}, y) dx_{i1} dx_{j1} dx_{i2} dx_{j2} dy
\end{equation*}

for $\eta:\R^{3}\rightarrow\R$ is a bounded function. %
A general estimation procedure will be given in
Section~\ref{sec:taylor:Estimation-of-quadratic}, extending the method
developed by~\cite{daveiga2013efficient}. %

The estimation of this term suggests to plug-in a preliminar estimator
of $f$, namely $\hat{f}$ (built on an independent sample of size
$n_1$), then an estimator of the quadratic functional. %
We chose to consider projection type estimators onto a finite subset,
$M_n$, of functional basis, $p_l(x_i, x_j, y)$, for $l \geq 1$. %
This leads to the final estimator that will be studied in this paper

\begin{multline}
  \label{eq:taylor:full_estimator}
  \hat{\sigma}_{ij} =
  \frac{1}{n_{2}}\sum_{k=1}^{n_{2}}H_{1}\bigl(\hat{f}, X_{ik}, X_{jk},
  Y_k\bigr) + \frac{1}{n_{2}(n_{2}-1)} \sum_{l\in M_n} \sum_{k\neq
    k^{\prime}=1}^{n_{2}}p_{l} \bigl(X_{ik}, X_{jk}, Y_{k}\bigr) \\
  \int p_{l} \bigl(x_{i}, x_{j}, Y_{k^{\prime}}\bigr)
  H_{3}\bigl(\hat{f}, x_{i}, x_{j}, X_{ik^{\prime}},
  X_{jk^{\prime}}, Y_{k^{\prime}}\bigr) \, dx_{i} \, dx_{j}\\
  - \frac{1}{n_{2}(n_{2}-1)} \sum_{l, l^{\prime}\in M_n} \sum_{k\neq
    k^{\prime}=1}^{n_{2}} p_{l}\bigl(X_{ik}, X_{jk}, Y_{k}\bigr)
  p_{l^{\prime}}\bigl(X_{ik^{\prime}}, X_{jk^{\prime}},
  Y_{k^{\prime}} \bigr)\\
  \int p_{l}\bigl(x_{i1}, x_{j1}, y\bigr) p_{l^{\prime}}\bigl(x_{i2},
  x_{j2}, y\bigr) H_{2}\bigl(\hat{f}, x_{i1}, x_{j2}, y\bigr) \,
  dx_{i1} \, dx_{j1} \, dx_{i2} \, dx_{j2}dy.
\end{multline}

where $H_{3}(f, x_{i1}, x_{j1}, x_{i2}, x_{j2}, y) = H_{2}(f, x_{i1},
x_{j2}, y)+H_{2}(f, x_{i2}, x_{j1}, y)$ and $n_{2}=n-n_{1}$. %
Note that the term, $\Gamma_{n}$, of Equation
\eqref{eq:taylor:T-ij-Taylor} is a remaining term that will be shown
to be negligible compared to the other terms.

We will prove that under some smoothness conditions, the only term
that drives the asymptotic properties of this estimator is $H_1$,
which will ensure the efficiency of the estimation procedure.

\section{Main Results}

In this section we provide the asymptotic behavior of the estimator
for $\sigma_{ij}$. %
Recall that in Section \ref{sec:taylor:Methodology} we constructed the
functional operator, $T_{ij}(f)$, by plugging-in a preliminary
estimator of the joint density, $f(x_i, x_j, y)$, into Equation
\eqref{eq:taylor:T-ij-Taylor}. %
However, in Equation \eqref{eq:taylor:full_estimator}, we have
introduced the functional basis, $p_l(x_i, x_j, y)$, which allows an
estimate of $\sigma_{ij}$ using a projection onto a finite subset,
$M_n$. %

In the next Subsection, we will describe some regularity conditions on
these components in order to ensure the convergence of our
estimator. %
Moreover, in Subsection \ref{sub:taylor:Efficient-Estimation-of} we
shall prove the asymptotic normality and efficiency of
$\hat{\sigma}_{ij}$. %

\subsection{Notations and Assumptions}
\label{sub:taylor:Hypothesis-and-Assumptions}


Consider the following notations. %
Let $a$ and $b$ be real numbers where $a<b$. %
Let, for a fixed $i$ and $j$, $\mathbb{L}^{2}(\, dx_{i}\, dx_{j}\,
dy)$ be the squar integrable functions in the cube, $[a, b]^3$. %
Moreover, let $(p_{l}(x_{i}, x_{j}, y))_{l\in \N}$ be an orthonormal
basis of $\mathbb{L}^{2}(\, dx_{i}\, dx_{j}\, dy)$. %
Let us denote $a_{l}=\int p_{l}f$ as the scalar product of $f$ with
$p_{l}$. %

Moreover, we can decompose the basis, $p_l(x_i, x_j, y)$, as
$\alpha_{l_{\alpha}}(x_{i}, x_{j}) \beta_{l_{\beta}}(y)$ with
$l_{\alpha}, l_{\beta} \in \N$. %
The set of functions, $\alpha_{l_{\alpha}}(x_{i}, x_{j})$ and
$\beta_{l_{\beta}}(y)$, are orthonormal bases in
$\mathbb{L}^{2}(dx_{i}\, dx_{j})$ and $\mathbb{L}^{2}(dy)$,
respectively.


We considered the following subset of $\mathbb{L}^{2}(\, dx_{i}\,
dx_{j}\, dy)$
\begin{equation*}
  \mathcal{E} = \left\{\sum_{l\in \N}e_{l}p_{l},\ \text{such as}\
    (e_{l})_{l\in \N} \ \text{satisfies} \ \sum_{l\in
      \N}\left|\frac{e_{l}}{c_{l}}\right|^{2}<1\right\}
\end{equation*}
where $(c_{l})_{l\in \N}$ is a decreasing fixed sequence. %

Assume that the triplet, $(X_{i}, X_{j}, Y)$, has a bounded joint
density, $f$, on $[a, b]^3$. %
Moreover, we suppose that $f$ belongs to the ellipsoid,
$\mathcal{E}$. %

Furthermore, $X_{n}{\overset{\mathcal{D}}{\longrightarrow}} X$
(resp. $X_{n}{\overset{\mathcal{P}}{\longrightarrow}} X$) denotes the
convergence \emph{in distribution} or \emph{weak} convergence
(resp. convergence \emph{in probability}) of $X_{n}$ to $X$. %
Additionally, we denote the support of $f$ as $\supp f$. %

For a fixed $n$, we chose a finite subset of indices belonging to $\N$
called $\left(M_{n}\right)_{n\geq 1}$. %
This sequence of indices, $M_n$, increases as $n$ increases. %
Also, $\vert M_{n}\vert$ will represent the cardinal of $M_{n}$. %

We define the partial projection of $f$ into the basis, $p_l$, with
only $M_n$ elements as $S_{M_{n}}f=\sum_{l\in M_{n}}a_{l}p_{l}$. %

We shall make three main assumptions:

\begin{assumption}
  \label{ass:taylor:A1}
  For all $n\geq 1$ there is a finite subset, $M_{n}\subset \N$, such
  that

  \begin{equation*}
    \sup_{l\notin M_{n}} \vert c_{l}\vert ^{2} \approx
    \sqrt{\vert{M_{n}}\vert}/n
  \end{equation*}

  ($A_{n}\approx B$ means $\lambda_{1} \leq A_{n}/B_n \leq
  \lambda_{2}$ for some positives constants $\lambda_{1}$ and
  $\lambda_{2}$). %
\end{assumption}

\begin{assumption}
  \label{ass:taylor:A2}We assume that $\supp f\subset[a, b]^3$ and for
  all $(x, y, z)\in\supp f$, $0<c_1\leq f(x, y, z)\leq c_2$ with $c_1,
  c_2\in\R$. %
\end{assumption}

\begin{assumption}\label{ass:taylor:A3}
  It is possible to find an estimator $\hat{f}$ of $f$ built with
  $n_{1}\approx n/\log\left(n\right)$ observations, such that for
  $\epsilon>0$,

  \begin{equation*}
    \forall(x, y, z)\in\supp f, \ 0<c_1 - \epsilon \leq \hat{f}(x, y,
    z) \leq c_2 + \epsilon
  \end{equation*}

  and,

  \begin{equation*}
    \mathbb{E}_{f}\Vert{\hat{f} - f}\Vert_{3}^{6}\leq Cn_{1}^{-6\lambda}
  \end{equation*}

  for some $\lambda>1/6$ and some constant $C$, not depending on $f$
  belonging to the ellipsoid $\mathcal{E}$. %
\end{assumption}

Assumption \ref{ass:taylor:A1} is necessary to bound the bias and
variance of $\hat{\theta}_{n}$. %
This condition allows us to control the mean square error of the queue
when we estimate $f$ by $S_{M_n}f$. %
In other words, Assumption \ref{ass:taylor:A1} balances the growing
size of the set, $M_n$, with the decay rate of the sequence, $c_l$. %
This behavior is strongly related to the smoothness of the density
function, $f$, and the size of the coefficients, $c_l$. %



One can find examples of a functional basis that satisfies Assumption
\ref{ass:taylor:A1}. %
For instance, we used the Example {2} from
\citet{laurent1996efficient} and its notation. %
Assume that $f$ belongs to some H\"{o}lder space with an index greater
than $s$. %
If a wavelet basis, $\tilde{\psi}$, has regularity, $r>s$, then
$f\in\mathcal{E} $ where

\begin{equation*}
  \mathcal{E} = \left\{\sum_{j\geq 0} \sum_{\lambda\in\Lambda_j}
    a(\lambda) \tilde{\psi}_\lambda : \ \text{where}\ \sum_{j\geq 0}
    \sum_{\lambda\in\Lambda_j} 2^{2js}\vert a(\lambda)\vert^2 \leq 1
  \right\}.
\end{equation*}

See \citet{meyer1992wavelets} for further details. %

Moreover, if $s>p/4$ and

\begin{equation*}
  M_n = \left\lbrace \lambda \in \Lambda_j, j\leq j_0,
    2^{j_0}=n^{2/(p+4s)} \right\rbrace
\end{equation*}

then $\sup_{l\notin M_{n}} \vert{c_{l}}\vert^{2} \approx
\sqrt{\vert{M_{n}}\vert}/n$. %
Also, $\vert M_n\vert/n \to 0$ with

\begin{equation*}
  \vert M_n\vert \approx n^{2p/(d+4s)}, \quad \sup_{l\notin
    M_{n}}\vert{c_{l}}\vert^{2} \approx 2^{-2j_0^2}=n^{-4s/(p+4s)}.
\end{equation*}


Assumption \ref{ass:taylor:A2} and \ref{ass:taylor:A3} help to
establish that $\Gamma_{n}=O(1/n)$, i.e. the error term in Equation
\eqref{eq:taylor:T-ij-Taylor} is negligible. %

Notice that in Assumption \ref{ass:taylor:A3}, the function,
$\hat{f}$, converges to $f$ faster than some given rates. %
This property is strongly related to the fact that the joint density
function, $f$, is regular enough.



For instance, for $\bm{x} = (x_1,x_2,x_3)\in\R^{3}$, $s>0$ and $L>0$
we defined the Nikol'skii class of functions, $\mathcal{H}_{3}(s, L)$,
as the set of functions, $f:\R^3\to \R$, whose partials derivatives up
to order $r=\left\lfloor s\right\rfloor $ inclusive exist and for
$l=1,2,3$ satisfy %
\begin{equation*}
  \left\Vert{ \frac{\partial^r f(\bm{x}+\bm{h})}{\partial x_l^r} -
      \frac{\partial^r f(\bm{x})}{\partial x_l^r} } \right\Vert_{3}
  \leq L  \vert  h^{s-r}\vert \ \text{where}
  \ \bm{h} = (h,h,h)  \ \forall h\in\R. %
\end{equation*}

Therefore, if $f$ belongs to $\mathcal{H}_{3}(s, L)$ with $s>3/4$ then
Assumption~\ref{ass:taylor:A3} is satisfied. %
The proof of this assertion can be found in the work of
\cite{ibragimov1983estimation,ibragimov1984more}. %
We refer the reader to \cite{laurent1993fonctionnelles} for more
examples on a regular class of functions that satisfy Assumptions
\ref{ass:taylor:A1}, \ref{ass:taylor:A2}, and \ref{ass:taylor:A3}. %

\subsection{Asymptotic behavior of the estimator of $\sigma_{ij}$ and
  of the conditional covariance matrix}
\label{sub:taylor:Efficient-Estimation-of}

We have two separate asymptotics with respect to $n_1$ and $n_2$. %
Both are allowed to grow to infinity as the the size of the sample,
$n$, grows larger. %
Therefore in the theorems, the asymptotics will be given with respect
to $n$.

The following theorem provides the asymptotic behavior of
$\hat{\sigma}_{ij}$ for $i$ and $j$. %
\begin{thm}
  \label{thm:taylor:Asymp-norm-Tf-ij} Let Assumptions
  \ref{ass:taylor:A1}-\ref{ass:taylor:A3} hold and $\vert
  M_{n}\vert/n\to 0$ when $n\rightarrow\infty$. %
  Then,
  \begin{equation}
    \label{eq:taylor:Asymp-norm-Tf-ij}
    \sqrt{n} \bigl(\hat{\sigma}_{ij} - \sigma_{ij}\bigr) \overset{\mathcal{D}}{\longrightarrow} \mathcal{N} (0,
    {C_{ij}(f)}),
  \end{equation}
  and
  \begin{equation}
    \label{eq:taylor:Var-Tf-ij}
    \lim_{n\to\infty} n\, \E[{\hat{\sigma}_{ij} -
      \sigma_{ij}}]^{2} = C_{ij}(f),
  \end{equation}
  where
  \[
  C_{ij}(f) = \V({H_{1}(f, X_{i}, X_{j}, Y)})
  \]
\end{thm}

We can see that the asymptotic variance of $\sigma_{ij}$ depends only
on $H_{1}(f, X_{i}, X_{j}, Y)$. %
In other words, the linear part of Equation
\eqref{eq:taylor:T-ij-Taylor} controls the asymptotic normality of
$\sigma_{ij}$. %
This property entails the natural efficiency of $\hat{\sigma}_{ij}$. %

The next theorem produces the $\sigma_{ij}$'s semi parametric
Cram\'{e}r-Rao bound. %

\begin{thm}[Semiparametric Cram\'{e}r Rao bound]
  \label{thm:taylor:Efficiency-T-ij}
  Consider the estimation of
  \begin{equation*}
    \sigma_{ij} = \E[{\E[{X_i\vert Y}]\E[{X_j\vert Y}]^{\top}}]
  \end{equation*}
  for a random vector, $(X_i, X_j, Y)$, with joint density,
  $f\in\mathcal{E}$.

  Let $f_{0}\in\mathcal{E}$ be a density verifying the assumptions of
  Theorem~\ref{thm:taylor:Asymp-norm-Tf-ij}. %
  Then, for any estimator, $\hat{\sigma}_{ij}$, of $\sigma_{ij}$ and
  every family, $\{ \mathcal{V}_{r}(f_{0})\}_{r>0}$, of neighborhoods
  of $f_{0}$ we have
  \[
  \inf_{\{\mathcal{V}_{r}(f_{0})\}_{r>0}} \liminf_{n\to\infty}
  \sup_{f\in\mathcal{V}_{r}(f_{0})} n\, \mathbb{E}_f
  \bigl[\hat{\sigma}_{ij} - \sigma_{ij}\bigr]^{2}\geq C_{ij}(f_{0})
  \]
  where $\mathcal{V}_{r}(f_{0}) = \left\{ f:\Vert{f -
      f_{0}}\Vert_{2}<r\right\}$ for $r>0$.
\end{thm}

Theorems \ref{thm:taylor:Asymp-norm-Tf-ij} and
\ref{thm:taylor:Efficiency-T-ij} establish the asymptotic efficiency
of the estimator, $\hat{\sigma}_{ij}$, defined in Equation
\eqref{eq:taylor:full_estimator} which we summarize in the next
Corollary. %

\begin{cor}
  Let the assumptions of Theorem \ref{thm:taylor:Asymp-norm-Tf-ij} and
  \ref{thm:taylor:Efficiency-T-ij} hold. %
  Then, the estimator, $\hat{\sigma}_{ij}$, defined in Equation
  \eqref{eq:taylor:full_estimator} is asymptotically efficient.
\end{cor}


We have proved asymptotic normality entry by entry of the matrix,
$\Sigma = (\sigma_{ij})_{p \times p}$, using the estimator,
$\widehat{\Sigma} = (\hat{\sigma}_{ij})_{p \times p}$, defined in
Equation \eqref{eq:taylor:full_estimator}. %
To extend the result to the whole matrix, we introduce the
half-vectorization operator, \emph{$\vech$}. %
This operator stacks only the columns from the principal diagonal of a
square matrix downwards in a column vector. %
Formally, for a $p\times p$ matrix, $ A = (a_{ij})$,
\[
\vech(A) = \left[a_{11}, \cdots, a_{p1}, a_{22}, \cdots, a_{p2},
  \cdots, a_{33}, \cdots, a_{pp} \right]^{\top}. %
\]
Let $\boldsymbol{H_{1}(f)}$ be the matrix with entries defined by
$\left(H_{1}(f_{ij}, x_{i}, x_{j}, y)\right)_{i, j}$ if $i$ is
different from $j$ and $\left(H_{1}(f_{i}, x_{i}, x_{i}, y)\right)_{i,
  i}$ when $i$ is equal to $j$ and $i, j=1, \ldots, p$. %

Corollary \ref{cor:taylor:Asymp-norm-Tf} generalizes our previous
results to the vectorial case.

\begin{cor}
  \label{cor:taylor:Asymp-norm-Tf}
  Let Assumptions \ref{ass:taylor:A1}-\ref{ass:taylor:A3} hold and
  $\vert{M_{n}}\vert/n\to 0$ when $n\rightarrow\infty$. %
  Then ${\boldsymbol{\widehat{T}}}_n$ has the following properties:
  \begin{gather*}
    \sqrt{n} \vech \left(\widehat{\Sigma} - \Sigma \right)
    \xrightarrow[]{\mathcal{D}} \mathcal{N}(0, {\boldsymbol
      C(f)}), \label{eq:taylor:Asymp-norm-Tf}\\
    \lim_{n\rightarrow\infty} n\, \E\left[\vech \left(\widehat{\Sigma}
        - \Sigma\right) \vech\left(\widehat{\Sigma} -
        \Sigma\right)^\top \right]
    = \bm{C}(f)\label{eq:taylor:Var-Tf}\\
    \intertext{where the limit is taking element-wise and} \boldsymbol
    C(f) = \C({\vech(\boldsymbol{H_{1}}(f))}).\nonumber
  \end{gather*}
\end{cor}

The estimator, $\hat{\sigma}_{ij}$, is asymptotically normal with a
variance depending on the linear term of the Taylor development. %
Given this particular nature, it was possible to show the asymptotic
efficiency of $\hat{\sigma}_{ij}$. %
Therefore, among all of the estimators of $\sigma_{ij}$, the estimator
defined in Equation \eqref{eq:taylor:full_estimator} has the lowest
variance. %
The conclusions in Theorems \ref{thm:taylor:Asymp-norm-Tf-ij} and
\ref{thm:taylor:Efficiency-T-ij} depend on a precise estimation of the
quadratic term of $\hat{\sigma}_{ij}$, which is handled in the
following section. %

\section{Estimation of quadratic functionals}
\label{sec:taylor:Estimation-of-quadratic}

We have proved, in Section \ref{sub:taylor:Efficient-Estimation-of},
the asymptotic normality and found an efficient semi parametric
Cram\'{e}r-Rao bound of the estimator, $\hat{\sigma}_{ij}$ defined in
Equation \eqref{eq:taylor:full_estimator}. %
We used the Taylor decomposition \eqref{eq:taylor:T-ij-Taylor} to
construct the estimator $\hat{\sigma}_{ij}$. %
In the present section, we will build an estimator for the quadratic
term
\begin{equation*}
  \int H_{2}(\hat{f}, x_{i1}, x_{j2}, y) f(x_{i1}, x_{j1}, y)
  f(x_{i2}, x_{j2}, y) \, dx_{i1} \, dx_{j1} \, dx_{i2} \, dx_{j2} \, dy.
\end{equation*}
To this end, we built a general estimator of the parameter with the
form:
\[
\theta = \int\eta(x_{i1}, x_{j2}, y) f(x_{i1}, x_{j1}, y) f(x_{i2},
x_{j2}, y) \, dx_{i1} \, dx_{j1} \, dx_{i2} \, dx_{j2}\, dy,
\]
for $f\in\mathcal{E}$ and $\eta:\R^{3}\rightarrow\R$ is a bounded
function. %

Given $M=M_{n}$, a subset of $\N$, consider the estimator
\begin{multline}
  \label{eq:taylor:estim-non-linear}
  \hat{\theta}_{n} = \frac{1}{n(n-1)} \sum_{l\in M}\sum_{k\neq
    k^{\prime}=1}^{n} p_{l}(X_{ik}, X_{jk}, Y_{k})\\
  \int p_{l}(x_{i}, x_{j}, Y_{k^{\prime}}) \left(\eta(x_{i},
    X_{jk^{\prime}}, Y_{k^{\prime}}) + \eta(
    X_{ik^{\prime}}, x_{j}, Y_{k^{\prime}})\right) \, dx_{i} \, dx_{j}\\
  - \frac{1}{n(n-1)} \sum_{l, l^{\prime}\in M} \sum_{k\neq
    k^{\prime}=1}^{n} p_{l}(X_{ik}, X_{jk}, Y_{k}) p_{l^{\prime}}(X_{ik^{\prime}}, X_{jk^{\prime}}, Y_{k^{\prime}})\\
  \int p_{l}(x_{i1}, x_{j1}, y) p_{l^{\prime}}(x_{i2}, x_{j2}, y)
  \eta(x_{i1}, x_{j2}, y) \, dx_{i1} \, dx_{j1} \, dx_{i2} \, dx_{j2}
  \, dy.
\end{multline}

To simplify the presentation of Theorem
\ref{thm:taylor:Asymp-norm-Tf-ij}, we write $\psi(x_{i1}, x_{j1},
x_{i1}, x_{j2}, y) = \eta(x_{i1}, x_{j2}, y)+\eta(x_{i2}, x_{j1}, y)$
verifying
\begin{multline*}
  \int \psi(x_{i1}, x_{j1}, x_{i2}, x_{j2}, y) \, dx_{i1} \, dx_{j1}
  \, dx_{i2}
  \, dx_{j2} dy \\
  = \int \psi(x_{i2}, x_{j2}, x_{i1}, x_{j1}, y) \, dx_{i1} \, dx_{j1}
  \, dx_{i2} \, dx_{j2} \, dy.
\end{multline*}

With this notation we can simplify Equation
\eqref{eq:taylor:estim-non-linear} into
\begin{multline}
  \label{eq:taylor:estim-non-linear-simp}
  \hat{\theta}_{n} = \frac{1}{n(n-1)}\sum_{l\in M}\sum_{k\neq
    k^{\prime}=1}^{n} p_{l}(X_{ik}, X_{jk}, Y_{k}) \\
  \int p_{l}(x_{i}, x_{j}, Y_{k^{\prime}}) \psi(x_{i}, x_{j},
  X_{ik^{\prime}}, X_{jk^{\prime}}, Y_{k^{\prime}}) \, dx_{i} \, dx_{j}\\
  - \frac{1}{n(n-1)} \sum_{l, l^{\prime}\in M} \sum_{k\neq
    k^{\prime}=1}^{n}p_{l}(X_{ik}, X_{jk}, Y_{k}) p_{l^{\prime}}(X_{ik^{\prime}}, X_{jk^{\prime}}, Y_{k^{\prime}})\\
  \int p_{l}(x_{i1}, x_{j1}, y) p_{l^{\prime}}(x_{i2}, x_{j2}, y)
  \eta(x_{i1}, x_{j2}, y) \, dx_{i1} \, dx_{j1} \, dx_{i2} \, dx_{j2}
  \, dy.
\end{multline}

The bias of $\hat{\theta}$ is equal to
\begin{multline*}
  \label{eq:taylor:Bias_theta_quad-1}
  - \int(S_{M}f(x_{i1}, x_{j1}, y) - f(x_{i1}, x_{j1}, y))
  (S_{M}f(x_{i2}, x_{j2}, y) - f(x_{i2}, x_{j2}, y))\\
  \eta(x_{i1}, x_{j2}, y) \, dx_{i1} \, dx_{j1} \, dx_{i2} \, dx_{j2}
  \, dy.
\end{multline*}

The following Theorem gives an explicit bound for the
$\hat{\theta}_{n}$ variance. %
\begin{thm}
  \label{thm:taylor:quad-estimator}
  Let Assumption \ref{ass:taylor:A1} hold. %
  Then, if $\, \vert{M_{n}}\vert/n\rightarrow0$ when $n\rightarrow
  \infty$, then $\hat{\theta}_{n}$ has the following property

  \begin{equation*}
    \label{eq:taylor:Bound-MSE-quad}
    \left\vert{n\, \E\bigl[{\bigl(\hat{\theta}_{n} -
          \theta\bigr)^{2}}\bigr]-\Lambda(f, \eta)}\right\vert
    \le \gamma \left[\frac{\vert{M_{n}}\vert}{n} + \Vert{S_{M_{n}}f-f}\Vert_{2} +
      \Vert{S_{M_{n}}g - g}\Vert_{2}\right],
  \end{equation*}

  where $g(x_{i}, x_{j}, y) = \int f(x_{i2}, x_{j2}, y) \psi(x_{i},
  x_{j}, x_{i2}, x_{j2}, y) \, dx_{i2} \, dx_{j2}$, and
  \begin{multline*}
    \Lambda(f, \eta) = \int g(x_{i}, x_{j}, y)^{2} f(x_{i}, x_{j}, y)
    \, dx_{i} \, dx_{j} \, dy \\
    - \left(\int g(x_{i}, x_{j}, y) f(x_{i}, x_{j}, y) \, dx_{i} \,
      dx_{j} \, dy\right)^{2},
  \end{multline*}
  where $\gamma$ is a constant depending only on $\Vert f
  \Vert_{\infty}$, $\Vert{\eta}\Vert_{\infty}$ and $\Delta = (b -
  a)^2$. %
  Moreover, this constant is an increasing function of these
  quantities. %
\end{thm}

Note that Equation \eqref{eq:taylor:Bound-MSE-quad} implies that
\begin{equation*}
  \lim_{n\to\infty} n\, \E\bigl[\bigl(\hat{\theta}_{n} -
  \theta\bigr)^{2}\bigr] =  \Lambda(f, \eta).
\end{equation*}

We can control the quadratic term of $\hat{\sigma}_{ij}$, which is a
particular case of $\theta$ choosing $\eta(x_{i1}, x_{j2}, y) =
H_{2}(\hat{f}, x_{i1}, x_{j2}, y)$. %

We will show, in Proof \ref{proof:thm-asymp-tf-ij}, that $\Lambda(f,
\eta)\to0$ when $n\to \infty$. %
Consequently, the linear part of $\hat{\sigma}_{ij}$ governs its
asymptotic variance, which also yields asymptotic efficiency. %

\section{Conclusion}

In this paper, we proposed a new way to estimate
$\C({\E[{\boldsymbol{X}\vert Y}]})$, which is different from the usual
plug-in type of estimators. %
We used a general functional, $T_{ij}(f)$, depending on the joint
density function, $f$, of $(X_i, X_j, Y)$. %
In particular, we chose a suitable approximation, $\hat{f}$, of $f$
and constructed a coordinate-wise Taylor expansion around it up to the
third order. %
We call this estimator $\hat{\sigma}_{ij}$. %
This expansion serves to estimate $\C({\E[{\boldsymbol{X}\vert Y}]})$
using an orthonormal base of $\mathbb{L}^2(\, dx_i\, dx_j\, dy)$. %

We highlighted that $\hat{\sigma}_{ij}$ is asymptoticly normal with a
variance lead by the first order term. %
This behavior also causes efficiency from the Cram\'{e}r-Rao's point
of view. %
Again, the Cram\'{e}r-Rao bound depends only on the linear part of the
Taylor series.

With the help of the $\vech$ operator, we expanded our results to the
matrix estimator, $\widehat{\Sigma}$, formed with the entries,
$\hat{\sigma}_{ij}$. %
We showed that the $\boldsymbol{T}(f)$'s linear term guides the
variance for the $\widehat{\Sigma}$'s asymptotic normality. %

Even if we had principally aimed to study a new class of estimators
for $\C({\E[{\boldsymbol{X}\vert Y}]})$, we referred to
\citet{daveiga2013efficient} for some simulations in a context similar
to ours. %
In general, their numerical result behaves reasonably well despite the
complexity of its implementation. %
These results could also work in our framework and we will consider
them in a future article. %

The estimator, $\widehat{\Sigma}$, could have negative eigenvalues,
violating the semi positive definiteness of the covariance. %
From a practical point of view, we could project $\widehat{\Sigma}$
into the space of positive semi definite matrices. %
Therefore, we would first diagonalize $\widehat{\Sigma}$ and then
replace negative eigenvalues by 0. %
The resulting estimator is then semi positive definite. %
The works of \cite{bickel2008thresholding,bickel2008regularized} and
\cite{cai2010optimal}, present an extended discussion about techniques
on matrix regularization. %

This research constitutes a first step in the study of estimators
based on a Taylor series with minimum variance. %
To simplify the complex implementation of this estimator, we will
explore another kind of technique like nonparametric methods for
example. %





\section{Appendix}

\subsection{Proofs }
\label{sec:taylor:Proofs}

\begin{proof}[\textbf{Proof of Proposition
    \ref{prop:taylor:Decomposition_Taylor_Tijf}}]
  \textbf{}

  We need to calculate the first three derivatives of $F(u)$. %
  To ease the calculations, notice first that
  \begin{equation}
    \label{eq:taylor:m-i-u-prime}
    \frac{d}{du}{m_{i}(f_{u}, y)}u
    = \frac{\int \Bigl(x_{i} - m_{i}(f_{u}, y)\Bigr)\Bigl(f(x_{i},
      x_{j}, y) - \hat{f}(x_{i}, x_{j}, y)\Bigr) \, dx_{i} \,
      dx_{j}}{\int f_{u}(x_{i}, x_{j}, y) \, dx_{i} \, dx_{j}}.
  \end{equation}



  It is possible to interchange the derivate with the integral sign
  because $f$ and $ \hat{f} $ are bounded. Now, using Equation
  \eqref{eq:taylor:m-i-u-prime} and taking $u=0$ we have

  \begin{multline}
    \label{eq:taylor:F-prime-0}
    F^{\prime}(0) = \int\left[x_{i}m_{j}
      (\hat{f}, y)+x_{j}m_{i}(\hat{f}, y) - m_{i}(\hat{f}, y)m_{j}(\hat{f}, y)\right]\\
    \Bigl(f(x_{i}, x_{j}, y) - \hat{f}(x_{i}, x_{j}, y)\Bigr) \,
    dx_{i}\, dx_{j}\, dy.
  \end{multline}

  Deriving $m_{i}(f_{u}, y)m_{j}(f_{u}, y)$ using the same arguments
  as in Equation \eqref{eq:taylor:m-i-u-prime} and again taking $u=0$
  we get,

  \begin{multline}
    \label{eq:taylor:F-prime-prime-0}
    F^{\prime\prime}(0) = \int\frac{2}{\int\hat{f}(x_{i}, x_{j}, y) \,
      dx_{i} \, dx_{j}} \Bigl(x_{i1} - m_{i}(\hat{f}, y)) (x_{j2} -
    m_{j}(\hat{f},
    y)\Bigr) \\
    \Bigl(f(x_{i1}, x_{j1}, y) - \hat{f}(x_{i1}, x_{j1}, y)\Bigr) \\
    \Bigl(f(x_{i2}, x_{j2}, y) - \hat{f}(x_{i2}, x_{j2}, y)\Bigr) \,
    dx_{i1} \, dx_{j1} \, dx_{i2} \, dx_{j2} \, dy.
  \end{multline}

  Using the previous arguments we also find that
  \begin{multline}
    \label{eq:taylor:F-prime-prime-prime}
    F^{\prime\prime\prime}(u) = \int\frac{-6}{\int f_{u}(x_{i}, x_{j},
      y) dx_{i} dx_{j}} \Bigl(x_{i1}
    - m_{j}(f_{u}, y)\Bigr) \Bigl(x_{j2} - m_{j}(f_{u}, y)\Bigr)\\
    \Bigl(f(x_{i1}, x_{j1}, y) - \hat{f}(x_{i1}, x_{j1}, y)\Bigr)
    \Bigl(f(x_{i2}, x_{j2}, y) - \hat{f}(x_{i2}, x_{j2}, y)\Bigr)\\
    \Bigl(f(x_{i3}, x_{j3}, y) - \hat{f}(x_{i3}, x_{j3}, y)\Bigr) \,
    dx_{i1} \, dx_{j1} \, dx_{i2} \, dx_{j2} \, dx_{i3} \, dx_{j3} \,
    dy
  \end{multline}
  Replacing Equation \eqref{eq:taylor:F-prime-0},
  \eqref{eq:taylor:F-prime-prime-0}, and
  \eqref{eq:taylor:F-prime-prime-prime} into Equation
  \eqref{eq:taylor:F-Taylor} we get the desired decomposition.
\end{proof}

\begin{proof}[\textbf{Proof of Theorem
    \ref{thm:taylor:Asymp-norm-Tf-ij}}]
  \label{proof:thm-asymp-tf-ij}

  We will first control the remaining term of Equation
  \eqref{eq:taylor:Gamma},
  \[
  \Gamma_{n} = \frac{1}{6} F^{\prime\prime\prime}(\xi)(1-\xi)^{3}.
  \]
  Remember that
  \begin{align*}
    F^{\prime\prime\prime}(\xi) = & \ -6 \int \frac{\Bigl(x_{i1} -
      m_{i}(f_{\xi}, y)\Bigl) \Bigl(x_{j2}-m_{j}(f_{\xi},
      y)\Bigr)}{\Bigl(\int
      f_{\xi}(x_{i}, x_{j}, y) \, dx_{i} \, dx_{j} \Bigr)^{2}} \\
    & \left(f(x_{i1}, x_{j1}, y) - \hat{f}(x_{i1}, x_{j1}, y)\right)
    \left(f(x_{i2}, x_{j2}, y) - \hat{f}(x_{i2}, x_{j2}, y)\right)\\
    & \left(f(x_{i3}, x_{j3}, y) - \hat{f}(x_{i3}, x_{j3}, y)\right)
    \, dx_{i1} \, dx_{j1} \, dx_{i2} \, dx_{j2} \, dx_{i3} \, dx_{j3
    }\, dy,
  \end{align*}
  Assumptions \ref{ass:taylor:A1} and \ref{ass:taylor:A2} ensure that
  the first part of the integrand is bounded by a constant,
  $\mu$. Furthermore,
  \begin{align*}
    \left\vert{\Gamma_{n}}\right\vert \leq & \ \mu \int \left\vert
      {f(x_{i1}, x_{j1}, y) - \hat{f}(x_{i1}, x_{j1}, y)} \right\vert
    \left\vert {f(x_{i2}, x_{j2}, y) -
        \hat{f}(x_{i2}, x_{j2}, y)} \right\vert\\
    & \left\vert {f(x_{i3}, x_{j3}, y) - \hat{f}(x_{i3}, x_{j3},
        y)}\right\vert \, dx_{i1}
    \, dx_{j1} \, dx_{i2} \, dx_{j2} \, dx_{i3} \, dx_{j3} \, dy\\
    = & \ \mu \int \left(\int \left\vert{f(x_{i}, x_{j}, y) -
          \hat{f}(x_{i}, x_{j}, y)} \right\vert
      \, dx_{i} \, dx_{j} \right)^{3} \, dy\\
    \leq & \ \mu \Delta^3 \int \left\vert{f(x_{i}, x_{j}, y) -\
        \hat{f}(x_{i}, x_{j}, y)}\right\vert^{3} \, dx_{i} \, dx_{j}
    \, dy
  \end{align*}
  by the H\"older inequality. Then, $\E[{\Gamma_{n}^{2}}]$ is equal to
  $O(\mathbb{E}\Vert f - \hat{f}\Vert_{3}^{6})$. %
  Since $\hat{f}$ verifies Assumption \ref{ass:taylor:A3}, this
  quantity is of order $O(n_{1}^{-6\lambda})$. %
  We also assume $n_{1}\approx n/\log(n)$ and $\lambda>1/6$, then
  $n_{1}^{-6\lambda}=o\left(1/n\right)$. %
  Therefore, we get $\E[{\Gamma_{n}^{2}}] = o(1/n)$ which implies that
  the remaining term, $\Gamma_{n}$, is negligible.

  To prove the asymptotic normality of $\hat{\sigma}_{ij}$, we shall
  show that ${\sqrt{n}{(\hat{\sigma}_{ij}-T_{ij}(f))}}$ and
  \begin{multline}
    \label{eq:taylor:def-Z-ij}
    Z_{ij}^{(n)} = \frac{1}{n_{2}} \sum_{k=1}^{n_{2}} H_{1}\bigl(f,
    X_{ik}, X_{jk}, Y_k\bigr) \\
    - \int H_{1}(f, x_{i}, x_{j}, y) f(x_{i}, x_{j}, y)) \, dx_{i} \,
    dx_{j} \, dy
  \end{multline}
  have the same asymptotic behavior. %
  We can get for $ Z_{ij}^{(n)}$ a classic central limit theorem with
  variance
  \begin{align*}
    C_{ij}(f) = & \  \V({H_{1}(f, x_{i}, x_{j}, y)})\\
    = & \ \int H_{1}(f, x_{i}, x_{j}, y)^{2} f(x_{i}, x_{j}, y)) \,
    dx_{i}
    \, dx_{j} \, dy \\
    & - \left(\int H_{1}(f, x_{i}, x_{j}, y) f(x_{i}, x_{j}, y) \,
      dx_{i} \, dx_{j} \, dy \right)^{2}
  \end{align*}
  which implies Equation \eqref{eq:taylor:Asymp-norm-Tf-ij} and
  \eqref{eq:taylor:Var-Tf-ij}. %
  In order to establish our claim, we will show that
  \begin{equation}
    \label{eq:taylor:def-R-ij}
    R_{ij}^{(n)} = \sqrt{n}\left[\hat{\sigma}_{ij} - T_{ij}(f)-
      Z_{ij}^{(n)} \right]
  \end{equation}
  has a second-order moment converging to 0. %

  Define $\widehat{Z}_{ij}^{(n)}$ as $ Z_{ij}^{(n)}$ with $f$ replaced
  by $\hat{f}$. %
  Let us note that $ R_{ij}^{(n)}=R_{1}+R_{2}$ where
  \begin{align*}
    R_{1} & = \sqrt{n} \left[\hat{\sigma}_{ij} - T_{ij}(f) -
      \widehat{Z}_{ij}^{(n)} \right]\\
    R_{2} & =\sqrt{n} \left[\widehat{Z}_{ij}^{(n)} - Z_{ij}^{(n)}
    \right].
  \end{align*}
  It only remains to state that $\E[{R_{1}^{2}}]$ and
  $\E[{R_{2}^{2}}]$ converges to 0. %
  We can rewrite $R_{1}$ as
  \[
  R_{1} = -\sqrt{n} \left[\widehat{Q} - Q + \Gamma_{n} \right]
  \]
  with
  \begin{gather*}
    Q = \int H_{2}(\hat{f}, x_{i1}, x_{j2}, y) f(x_{i1}, x_{j1}, y)
    f(x_{i2}, x_{j2}, y) \, dx_{i1} \, dx_{j1} \, dx_{i2} \, dx_{j2} \, dy\\
    H_{2}(\hat{f}, x_{i1}, x_{j2}, y) = \frac{1}{\int\hat{f}(x_{i},
      x_{j}, y) \, dx_{i} \, dx_{j}} \left(x_{i1} - m_{i}(\hat{f},
      y)\right) \left(x_{j2} - m_{j}(\hat{f}, y) \right).
  \end{gather*}

  We can estimate $ H_{2}(\hat{f}, x_{i1}, x_{j2}, y) = \eta(x_{i1},
  x_{j2}, y)$ as done in
  Section~\ref{sec:taylor:Estimation-of-quadratic}. %
  Let $\widehat{Q}$ be the estimator of $ Q $. %
  Since $\E[{\Gamma_{n}^{2}}] = o(1/n)$, we only have to control the
  Term, $\sqrt{n}(\widehat{Q} - Q)$, such that $\lim_{n\to\infty} n\,
  \E[{\widehat{Q} - Q}]^{2}=0$ by Lemma \ref{lem:Asymp-n-hatQ-Q} in
  Section \ref{sec:taylor:Technical-Results}. %
  This Lemma implies that $\E[{R_{1}^{2}}]\to0$ as
  $n\rightarrow\infty$. %
  For $R_{2}$ we have
  \begin{multline*}
    \E[{R_{2}^{2}}] = \frac{n}{n_{2}} \left[\int \left(H_{1}(f, x_{i},
        x_{j}, y) - H_{1}(\hat{f}, x_{i}, x_{j}, y)
      \right)^{2}f(x_{i}, x_{j}, y)) \, dx_{i} \, dx_{j} \, dy \right]\\
    - \frac{n}{n_{2}}\left[\int H_{1}(f, x_{i}, x_{j}, y)
      f(x_{i}, x_{j}, y)) \, dx_{i} \, dx_{j} \, dy \right.\\
    \left. -\int H_{1}(\hat{f}, x_{i}, x_{j}, y)^{2} f(x_{i}, x_{j},
      y)) \, dx_{i} \, dx_{j} \, dy \right]{}^{2}.
  \end{multline*}

  The same arguments as the ones of Lemma \ref{lem:Asymp-n-hatQ-Q}
  (Mean Value Theorem and Assumptions \ref{ass:taylor:A2} and
  \ref{ass:taylor:A3}) show that $\E[{R_{2}^{2}}]\to 0$. %
\end{proof}

\begin{proof}[\textbf{Proof of Theorem \ref{thm:taylor:Efficiency-T-ij}} ]
  To prove the inequality we will use the usual framework described
  in~\citet{ibragimov1991asymptotically}. %
  The first step is to calculate the Fr\'echet derivative of
  $T_{ij}(f)$ at some point $f_{0}\in\mathcal{E}$. %
  Assumptions \ref{ass:taylor:A2} and \ref{ass:taylor:A3} and Equation
  \eqref{eq:taylor:T-ij-Taylor}, imply that
  \begin{align*}
    T_{ij}(f) - T_{ij}(f_{0}) & =\int \Bigl(x_{i}m_{j}(f_{0}, y) +x
    _{j}m_{i}(f_{0}, y) - m_{i}(f_{0}, y) m_{j}(f_{0}, y) \Bigr)\\
    & \Bigl(f(x_{i}, x_{j}, y) - f_{0}(x_{i}, x_{j}, y) \Bigr) \,
    dx_{i} \, dx_{j} \, dy + O\left(\int\left(f -
        f_{0}\right)^{2}\right)
  \end{align*}
  where $m_{i}(f_{0}, y) = \int x_{i} f_{0}(x_{i}, x_{j}, y) \, dx_{i}
  \, dx_{j} \, dy/\int f_{0}(x_{i}, x_{j}, y) \, dx_{i} \, dx_{j} \,
  dy$. %
  Therefore, the Fr\'echet derivative of $T_{ij}(f)$ at $f_{0}$ is
  $T_{ij}^{\prime}(f_{0})\cdot h =\left\langle H_{1}(f_{0}, \cdot),
    h\right\rangle $ with
  \begin{equation*}
    H_{1}(f_{0}, x_{i}, x_{j}, y) = x_{i} m_{j}(f_{0}, y) +
    x_{j}m_{i}(f_{0}, y) - m_{i}(f_{0}, y) m_{j}(f_{0}, y).
  \end{equation*}
  Using the results of \citet{ibragimov1991asymptotically}, we denote
  the set of functions in $\mathbb{L}^{2}(\, dx_{i}\, dx_{j}\, dy)$
  orthogonal to $\sqrt{f_{0}}$ as
  \begin{equation*}
    H(f_{0}) = \left\{u\in\mathbb{L}^{2}(\, dx_{i}\,
      dx_{j}\, dy), \int u(x_{i}, x_{j}, y) \sqrt{f_{0}(x_{i}, x_{j},
        y)} \, dx_{i} \, dx_{j} \, dy = 0\right\}.
  \end{equation*}

  Naming $\P_{H(f_{0})}$ as the projection onto $H(f_{0})$,
  $A_{n}(t)=\left(\sqrt{f_{0}}\right)t/\sqrt{n}$ and $P_{f_{0}}^{(n)}$
  is the joint distribution of $\bigl(X_{ik}, X_{jk}\bigr)\ k=1,
  \ldots, n$ under $f_{0}$. %
  Since $\bigl(X_{ik}, X_{jk}\bigr)\ k=1, \ldots, n$ are i.i.d., the
  family, $\bigl\{P_{f_{0}}^{(n)}, f\in\mathcal{E}\bigr\}$, is
  differentiable in a quadratic mean at $f_{0}$ and therefore locally
  asymptotically normal at all points, $f_{0}\in\mathcal{E}$, in the
  $H(f_{0})$ direction with a normalizing factor $A_{n}(f_{0})$ (see
  the details in~\citet{van2000asymptotic}). %
  Then, by the results of \citet{ibragimov1991asymptotically}, we say
  that under these conditions,
  denoting $K_{n} = B_{n}\theta^{\prime}(f_{0})A_{n}\P_{H(f_{0})}$
  with $B_{n} = \sqrt{n}u$, if
  $K_{n}\overset{\mathcal{D}}{\longrightarrow} K$ and if
  $K(u)=\left\langle t, u\right\rangle $, then for every estimator,
  $\hat{\sigma}_{ij}$, of $T_{ij}(f)$ and every family,
  $\mathcal{V}(f_{0})$, of vicinities of $f_{0}$, we have
  \begin{equation*}
    \inf_{\{\mathcal{V}(f_{0})\}} \liminf_{n\to\infty}
    \sup_{f\in\mathcal{V}(f_{0})} n\, \E[{\hat{\sigma}_{ij} -
      T_{ij}(f_{0})}]^{2} \geq \Vert t_{\mathbb{L}^{2}(\, dx_{i}\,
      dx_{j}\, dy)}\Vert^{2}.
  \end{equation*}

  Here,
  \begin{equation*}
    K_{n}(u) = \sqrt{n}T^{\prime}(f_{0}) \cdot
    \frac{\sqrt{f_{0}}}{\sqrt{n}} \P_{H(f_{0})}(u) = T^{\prime}(f_{0})
    \left(\sqrt{f_{0}}\left(u - \sqrt{f_{0}}\int
        u\sqrt{f_{0}}\right)\right),
  \end{equation*}
  since for any $u\in\mathbb{L}^{2}(\, dx_{i}\, dx_{j}\, dy)$ we can
  write it as $u = \sqrt{f_{0}} \left \langle \sqrt{f_{0}},
    u\right\rangle +\P_{H(f_{0})}(u)$. %
  In this case $K_{n}(u)$ does not depend on $n$ and
  \begin{align*}
    K(h) & =T^{\prime}(f_{0}) \cdot \left(\sqrt{f_{0}} \left(u -
        \sqrt{f_{0}} \int h\sqrt{f_{0}} \right)\right)\\
    & = \int H_{1}(f_{0}, \cdot) \sqrt{f_{0}}u - \int
    H_{1}(f_{0}, \cdot) \sqrt{f_{0}} \int u \sqrt{f_{0}}\\
    & = \left\langle t, u \right\rangle
  \end{align*}
  with
  \begin{equation*}
    t(x_{i}, x_{j}, y) = H_{1}(f_{0}, x_{i}, x_{j}, y) \sqrt{f_{0}} -
    \left(\int H_{1}(f_{0}, x_{i}, x_{j}, y)f_{0}\right) \sqrt{f_{0}}.
  \end{equation*}

  The semi parametric Cram\'{e}r-Rao bound for this problem is thus
  \begin{multline*}
    \Vert t_{\mathbb{L}^{2}(\, dx_{i}, \, dx_{j}, \, dy)}\Vert = \int
    H_{1}(f_{0}, x_{i}, x_{j}, y)^{2} f_{0} \, dx_{i} \, dx_{j} \, dy \\
    - \left(\int H_{1}(f_{0}, x_{i}, x_{j}, y) f_{0} \, \, dx_{i} \,
      dx_{j} \, dy \right)^{2}
  \end{multline*}
  and we recognize the expression, $C_{ij}(f_{0})$, found in Theorem
  \ref{thm:taylor:Asymp-norm-Tf-ij}.
\end{proof}

\begin{proof}[\textbf{Proof of Corollary
    \ref{cor:taylor:Asymp-norm-Tf}}]
  \textbf{}This proof is based on the following observation. Employing
  Equation \eqref{eq:taylor:def-R-ij} we have
  \[
  {\boldsymbol{\widehat{T}}}_n - \boldsymbol T(f) = {\boldsymbol
    Z}_n(f) + \frac{{\boldsymbol R}_n}{\sqrt{n}}
  \]
  where ${\boldsymbol Z}_n(f)$ and ${\boldsymbol R}_n$ are matrices
  with elements $ Z_{ij}^{(n)}$ and $ R_{ij}^{(n)}$, defined in
  Equation \eqref{eq:taylor:def-Z-ij} and \eqref{eq:taylor:def-R-ij},
  respectively.

  Hence we have,
  \begin{equation*}
    n\, \E\left[ \left\Vert \vech \left({\boldsymbol{\widehat{T}}}_n -
          \boldsymbol T(f) - {\boldsymbol Z}_n(f)\right)\right
      \Vert^{2} \right] = \E\left[ \left\Vert \vech\left({\boldsymbol
            R}_n\right)\right
      \Vert^{2} \right] = \sum_{i\leq j}\E\Bigl[\Bigl(
    R_{ij}^{(n)}\Bigr)^{2}\Bigr].
  \end{equation*}

  We see by Lemma \ref{lem:Asymp-n-hatQ-Q} that
  $\E[{R_{ij}^{2}}]\rightarrow0$ as $n\rightarrow0$. %
  It follows that
  \begin{equation*}
    n\, \E\left[ \left\Vert{\vech\left({\boldsymbol{\widehat{T}}}_n -
            \boldsymbol T(f) - {\boldsymbol Z}_n(f)\right)}\right
      \Vert^{2}\right]\rightarrow0 \text{ as } n\rightarrow0.
  \end{equation*}

  We know that if $X_{n}$, $X$ and $Y_{n}$ are random variables, then
  if $X_{n}\overset{\mathcal{D}}{\longrightarrow} X$ and
  $\left(X_{n}-Y_{n}\right)\overset{\mathcal{P}}{\longrightarrow}0, $
  it follows that $Y_{n}\overset{\mathcal{D}}{\longrightarrow} X$. %

  Remember also that the convergence in $\mathbb{L}^{2}$ implies
  convergence in the probability, therefore
  \[
  \sqrt{n}\vech \left({\boldsymbol{\widehat{T}}}_n - \boldsymbol T(f)
    - {\boldsymbol Z}_n(f)\right)
  \overset{\mathcal{P}}{\longrightarrow} 0.
  \]
  By the Multivariate Central Limit Theorem we have that
  $\sqrt{n}\vech{\left({\boldsymbol
        Z}_n(f)\right)}\overset{\mathcal{D}}{\longrightarrow}\mathcal{N}(0,
  {\boldsymbol C(f)})$. %
  Therefore, $\sqrt{n}\vech{\left({\boldsymbol{\hat{T}}}_n -
      \boldsymbol
      T(f)\right)}\overset{\mathcal{D}}{\longrightarrow}\mathcal{N}(0,
  {\boldsymbol C(f)})$.
\end{proof}

\begin{proof}[\textbf{Proof of Theorem
    \ref{thm:taylor:quad-estimator}}]
  \textbf{} For abbreviation, we write $M$ instead of $M_{n}$ and set
  $m=\vert{M_{n}}\vert$. %
  We first compute the mean squared error of $\hat{\theta}_{n}$ as
  \begin{equation*}
    \E\bigl[{\hat{\theta}_{n} - \theta}\bigr]^{2} =
    \B^2({\hat{\theta}_{n}})
    +     \V({\hat{\theta}_{n}})
  \end{equation*}
  where $\B({\hat{\theta}_{n}}){} = \E[{\hat{\theta}_{n}}] - \theta$.

  We begin the proof by bounding $\V({\hat{\theta}_{n}})$. %
  Let $A$ and $B$ be $m\times1$ vectors with components
  \begin{align*}
    a_{l} & = \int p_{l}(x_{i}, x_{j}, y) f(x_{i}, x_{j}, y) \, dx_{i}
    \, dx_{j} \, dy \quad l=1, \dots, m, \\
    b_{l} & = \int p_{l}(x_{i1}, x_{j1}, y) f(x_{i2}, x_{j2}, y)
    \psi(x_{i1}, x_{j1}, x_{i2}, x_{j2}, y) \, dx_{i1} \, dx_{j1} \,
    dx_{i2}
    \, dx_{j2} \, dy\\
    & = \int p_{l}(x_{i}, x_{j}, y) g(x_{i}, x_{j}, y) \, dx_{i} \,
    dx_{j} \, dy\quad l=1, \dots, m
  \end{align*}
  where $g(x_{i}, x_{j}, y) = \int f(x_{i2}, x_{j2}, y) \psi(x_{i},
  x_{j}, x_{i2}, x_{j2}, y) \, dx_{i2} \, dx_{j2}$. Let $Q$ and $R$ be
  $m\times1$ vectors of centered functions
  \begin{align*}
    q_{l}(x_{i}, x_{j}, y) & = p_{l}(x_{i}, x_{j}, y)-a_{l}\\
    r_{l}(x_{i}, x_{j}, y) & = \int p_{l}(x_{i2}, x_{j2}, y)
    \psi(x_{i}, x_{j}, x_{i2}, x_{j2}, y) \, dx_{i2} \, dx_{j2} -
    b_{l}
  \end{align*}
  for $l=1, \dots, m$. %
  Let $C$ be an $m\times m$ matrix of constants with indices $l,
  l^{\prime}=1, \dots, m$ defined by
  \[
  c_{ll^{\prime}} = \int p_{l}(x_{i1}, x_{j1}, y)
  p_{l^{\prime}}(x_{i2}, x_{j2}, y) \eta(x_{i1}, x_{j2}, y) \, dx_{i1}
  \, dx_{j1} \, dx_{i2} \, dx_{j2} \, dy.
  \]
  Let us denote $U_{n}$ by the process
  \[
  U_{n}h = \frac{1}{n(n-1)} \sum_{k\neq k^{\prime}=1}^{n} h\bigl(
  X_{ik}, X_{jk}, Y_k, X_{ik^{\prime}} , X_{jk^{\prime}},
  Y_{k^{\prime}}\bigr)
  \]
  and $P_{n}$ by the empirical measure
  \[
  P_{n}h = \frac{1}{n} \sum_{k=1}^{n} h \bigl(X_{ik}, X_{jk},
  Y_k\bigr)
  \]
  for some $h$ in $\mathbb{L}^{2}(\, dx_{i}, \, dx_{j}, \, dy)$. %
  With these notations, $\hat{\theta}_{n}$ has the Hoeffding's
  decomposition
  \begin{align*}
    & \hat{\theta}_{n} \\
    = & \ \frac{1}{n(n-1)} \sum_{l\in M} \sum_{k\neq k^{\prime}=1}^{n}
    \Bigl(q_{l}(X_{ik}, X_{jk}, Y_{k})+a_{l}\Bigr)
    \Bigl(r_{l}(X_{ik^{\prime}}, X_{jk^{\prime}}, Y_{k^{\prime}}) +
    b_{l}\Bigr) \\
    & - \frac{1}{n(n-1)} \sum_{l, l^{\prime}\in M} \sum_{k\neq
      k^{\prime}=1}^{n} \Bigl(q_{l}(X_{ik}, X_{jk}, Y_{k}) +
    a_{l}\Bigr) \Bigl(q_{l^{\prime}}(X_{ik^{\prime}},
    X_{jk^{\prime}}, Y_{k^{\prime}}) + a_{l^{\prime}}\Bigr)c_{ll^{\prime}}\\
    = & \ U_{n}K+P_{n}L + A^{\top}B - A^{\top}CA
  \end{align*}
  where
  \begin{align*}
    K\left(x_{i1}, x_{j1}, y_{1}, x_{i2}, x_{j2}, y_{2}\right) =
    & \  Q^{\top}(x_{i1}, x_{j1}, y_{1}) R(x_{i2}, x_{j2}, y_{2}) \\
    &  - Q^{\top}(x_{i1}, x_{j1}, y_{1}) C Q(x_{i2}, x_{j2}, y_{2})\\
    L(x_{i}, x_{j}, y) = & \ A^{\top} R(x_{i}, x_{j}, y) + B Q(x_{i},
    x_{j}, y) \\
    & \ -2 A^{\top}CQ(x_{i}, x_{j}, y).
  \end{align*}
  Therefore, $\V({\hat{\theta}_{n}}) = \V({U_{n}K}) + \V({P_{n}L}) - 2
  \C({U_{n}K, P_{n}L})$. %
  These three terms are bounded in Lemmas \ref{lem:Var-UK} -
  \ref{lem:Cov-UK-PL}, which gives
  \begin{equation*}
    \V({\hat{\theta}_{n}}) \leq \frac{20}{n(n-1)}
    \Vert{\eta}\Vert_{\infty}^{2}\Vert f\Vert_{\infty}^{2} \Delta^2(m+1)
    + \frac{12}{n} \Vert{\eta}\Vert_{\infty}^{2} \Vert
    f\Vert_{\infty}^{2} \Delta^2.
  \end{equation*}

  For $n$ large enough and a constant $\gamma\in\R$,
  \begin{equation*}
    \V({\hat{\theta}_{n}}) \leq \gamma \Vert{\eta}\Vert_{\infty}^{2}
    \Vert f\Vert_{\infty}^{2} \Delta^2 \left(\frac{m}{n^{2}} +
      \frac{1}{n}\right).
  \end{equation*}

  The term $\B({\hat{\theta}_{n}}){}$ is easily computed, as proven in
  Lemma \ref{lem:Bias-theta-hat-bound}, and is equal to
  \begin{multline*}
    -\int \left(S_{M}f(x_{i1}, x_{j1}, y) - f(x_{i1}, x_{j1},
      y)\right)
    \left(S_{M}f(x_{i2}, x_{j2}, y) - f(x_{i2}, x_{j2}, y)\right)\\
    \eta(x_{i1}, x_{j1}, x_{i2}, x_{j2}, y) \, dx_{i1} \, dx_{j1} \,
    dx_{i2} \, dx_{j2} \, dy.
  \end{multline*}

  From Lemma \ref{lem:Bias-theta-hat-bound}, we bound the bias of
  $\hat{\theta}_{n}$ by
  \begin{equation*}
    \vert{\B({\hat{\theta}_{n}}){}}\vert \le\Delta
    \Vert{\eta}\Vert_{\infty} \sup_{l\notin M} \vert{c_{l}}\vert^{2}.
  \end{equation*}
  The assumption of $\bigl( \sup_{l\notin M} \vert{c_{i}}
  \vert^{2}\bigr)^{2} \approx m/n^{2}$ and since $m/n\to0$, we deduce
  that $\E[{\hat{\theta}_{n}-\theta}]^{2}$ has a parametric rate of
  convergence, $O\left(1/n\right)$.

  Finally to prove Equation \eqref{eq:taylor:Bound-MSE-quad}, note
  that
  \begin{align*}
    n\, \E\bigl[{\hat{\theta}_{n} - \theta}\bigr]^{2} & = n
    \B^2({\hat{\theta}_{n}}) +
    n \V({\hat{\theta}_{n}})\\
    & = n\B^2({\hat{\theta}_{n}}) + n\V({U_{n}K}) + n\V({P_{n}L}).
  \end{align*}
  We previously proved that for some $\lambda_{1}, \lambda_{2}\in\R$
  \begin{align*}
    n\B^2({\hat{\theta}_{n}}) & \leq \lambda_{1}\Delta^2\Vert{\eta}\Vert_{\infty}^{2} \frac{m}{n}\\
    n\V({U_{n}K}) & \leq \lambda_{2}\Delta^2\Vert
    f\Vert_{\infty}^{2}\Vert{\eta}\Vert_{\infty}^{2}\frac{m}{n}.
  \end{align*}
  Thus, Lemma \ref{lem:Asymp-var-PL} implies
  \begin{equation*}
    \Bigl\vert{n\V({P_{n}L}) - \Lambda(f, \eta)}\Bigr\vert \leq
    \lambda\bigl[\Vert{S_{M}f - f}\Vert_{2} + \Vert{S_{M}g -
      g}\Vert_{2}\bigr],
  \end{equation*}
  where $\lambda$ is an increasing function of $\Vert
  f\Vert^{2}_{\infty}$, $\Vert{\eta}\Vert_{\infty}^{2}$ and
  $\Delta$. %
  Based on all of this we find Equation
  \eqref{eq:taylor:Bound-MSE-quad} which ends the proof of Theorem
  \ref{thm:taylor:quad-estimator}.
\end{proof}

\renewcommand{\bibname}{References} %
\bibliography{biblio_cond_cov_taylor} %

\begin{thebibliography}{}

\bibitem[Bickel and Levina, 2008a]{bickel2008thresholding}
Bickel, P.~J. and Levina, E. (2008a).
\newblock Covariance regularization by thresholding.
\newblock {\em The Annals of Statistics}, 36(6):2577--2604.

\bibitem[Bickel and Levina, 2008b]{bickel2008regularized}
Bickel, P.~J. and Levina, E. (2008b).
\newblock Regularized estimation of large covariance matrices.
\newblock {\em The Annals of Statistics}, 36(1):199--227.

\bibitem[Bura and Cook, 2001]{bura2001estimating}
Bura, E. and Cook, R.~D. (2001).
\newblock Estimating the structural dimension of regressions via parametric
  inverse regression.
\newblock {\em Journal of the Royal Statistical Society: Series B (Statistical
  Methodology)}, 63(2):393--410.

\bibitem[Cai et~al., 2010]{cai2010optimal}
Cai, T.~T., Zhang, C.-H., and Zhou, H.~H. (2010).
\newblock Optimal rates of convergence for covariance matrix estimation.
\newblock {\em The Annals of Statistics}, 38(4):2118--2144.

\bibitem[Cook and Ni, 2005]{cook2005sufficient}
Cook, R.~D. and Ni, L. (2005).
\newblock Sufficient dimension reduction via inverse regression.
\newblock {\em Journal of the American Statistical Association},
  100(470):410--428.

\bibitem[Da~Veiga and Gamboa, 2013]{daveiga2013efficient}
Da~Veiga, S. and Gamboa, F. (2013).
\newblock Efficient estimation of sensitivity indices.
\newblock {\em Journal of Nonparametric Statistics}, 25(3):573--595.

\bibitem[Duan and Li, 1991]{duan1991slicing}
Duan, N. and Li, K.-C. (1991).
\newblock Slicing regression: A link-free regression method.
\newblock {\em The Annals of Statistics}, 19(2):505--530.

\bibitem[Ferr\'{e} and Yao, 2005]{ferre2005smoothed}
Ferr\'{e}, L. and Yao, A. (2005).
\newblock Smoothed functional inverse regression.
\newblock {\em Statistica Sinica}, 15(3):665--683.

\bibitem[Ferr\'{e} and Yao, 2003]{ferre2003functional}
Ferr\'{e}, L. and Yao, A.~F. (2003).
\newblock Functional sliced inverse regression analysis.
\newblock 37(6):475--488.

\bibitem[Hardle and Tsybakov, 1991]{hardle1991slicedcomment}
Hardle, W. and Tsybakov, A.~B. (1991).
\newblock {Sliced Inverse Regression for Dimension Reduction: Comment}.
\newblock {\em Journal of the American Statistical Association}, 86(414):pp.
  333----335.

\bibitem[Hsing, 1999]{hsing1999nearest}
Hsing, T. (1999).
\newblock Nearest neighbor inverse regression.
\newblock {\em Annals of statistics}, 27(2):697--731.

\bibitem[Ibragimov and Khas'~minskii, 1983]{ibragimov1983estimation}
Ibragimov, I. and Khas'~minskii, R. (1983).
\newblock Estimation of distribution density.
\newblock {\em Journal of Soviet Mathematics}, 21(1):40--57.

\bibitem[Ibragimov and Khas'~minskii, 1984]{ibragimov1984more}
Ibragimov, I. and Khas'~minskii, R. (1984).
\newblock More on the estimation of distribution densities.
\newblock {\em Journal of Soviet Mathematics}, 25(3):1155--1165.

\bibitem[Ibragimov and Khas'Minskii, 1991]{ibragimov1991asymptotically}
Ibragimov, I.~A. and Khas'Minskii, R.~Z. (1991).
\newblock Asymptotically normal families of distributions and efficient
  estimation.
\newblock {\em The Annals of Statistics}, 19(4):1681--1724.

\bibitem[Laurent, 1996]{laurent1996efficient}
Laurent, B. (1996).
\newblock Efficient estimation of integral functionals of a density.
\newblock {\em The Annals of Statistics}, 24(2):659--681.

\bibitem[Laurent-Bonneau, 1993]{laurent1993fonctionnelles}
Laurent-Bonneau, B. (1993).
\newblock {\em Estimation de fonctionnelles integrales non lineaires d'une
  densite et de ses derivees}.
\newblock PhD thesis, Université Paris Sud, Centre d'Orsay.

\bibitem[Li, 1991a]{li1991sliced}
Li, K.-C. (1991a).
\newblock Sliced inverse regression for dimension reduction.
\newblock {\em J. Am. Stat. Assoc.}, 86(414):316--327.

\bibitem[Li, 1991b]{li1991rejoinder}
Li, K.-C. (1991b).
\newblock Sliced inverse regression for dimension reduction: Rejoinder.
\newblock {\em Journal of the American Statistical Association},
  86(414):337--342.

\bibitem[Meyer and Salinger, 1993]{meyer1992wavelets}
Meyer, Y. and Salinger, D.~H. (1993).
\newblock {\em Wavelets and Operators}.
\newblock Cambridge Studies in Advanced Mathematics ; 37. Cambridge University
  Press, Cambridge.

\bibitem[Setodji and Cook, 2004]{setodji2004kmeans}
Setodji, C.~M. and Cook, R.~D. (2004).
\newblock K -means inverse regression.
\newblock {\em Technometrics}, 46(4):421--429.

\bibitem[Van~der Vaart, 2000]{van2000asymptotic}
Van~der Vaart, A.~W. (2000).
\newblock {\em Asymptotic Statistics}, volume~3 of {\em Cambridge Series on
  Statistical and Probabilistic Mathematics}.
\newblock Cambridge University Press.

\bibitem[Zhu and Fang, 1996]{zhu1996asymptotics}
Zhu, L.-X. and Fang, K.-T. (1996).
\newblock Asymptotics for kernel estimate of sliced inverse regression.
\newblock {\em The Annals of Statistics}, 24(3):1053--1068.

\end{thebibliography}
\bibliographystyle{apalike}

\subsection{Technical Results}
\label{sec:taylor:Technical-Results}

\begin{lem}[Bias of $\hat{\theta}_{n}$]
  \label{lem:taylor:Bias-theta-hat} %
  The estimator $\hat{\theta}_{n}$ defined in
  \eqref{eq:taylor:estim-non-linear-simp} estimates $\theta$ with bias
  equal to
  \begin{multline*}
    -\int \Bigl(S_{M}f(x_{i1}, x_{j1}, y) - f(x_{i1}, x_{j1}, y)\Bigr)
    \Bigl(S_{M}f(x_{i2}, x_{j2}, y) - f(x_{i2}, x_{j2}, y)\Bigr)\\
    \eta(x_{i1}, x_{j2}, y) \, dx_{i1} \, dx_{j1} \, dx_{i2} \,
    dx_{j2} \, dy.
  \end{multline*}
\end{lem}
\begin{proof}[\textbf{Proof of Lemma \ref{lem:taylor:Bias-theta-hat}}]
  Let $\hat{\theta}_{n}=\hat{\theta}_{n}^{1}-\hat{\theta}_{n}^{2}$
  where
  \begin{align*}
    \hat{\theta}_{n}^{1} = & \ \frac{1}{n(n-1)} \sum_{l\in M}
    \sum_{k\neq k^{\prime}=1} p_{l}(X_{ik}, X_{jk}, Y_{k}) \\
    & \int p_{l}(x_{i}, x_{j}, Y_{k^{\prime}}) \psi(x_{i}, x_{j},
    X_{ik^{\prime}}, X_{jk^{\prime}}, Y_{k^{\prime}}) dx_{i} dx_{j} \\
    \hat{\theta}_{n}^{2} = & \ -\frac{1}{n(n-1)} \sum_{l,
      l^{\prime}\in M} \sum_{k\neq k^{\prime}=1}^{n} p_{l}(X_{ik},
    X_{jk}, Y_{k})
    p_{l^{\prime}} (X_{ik^{\prime}}, X_{jk^{\prime}}, Y_{k^{\prime}}) \\
    & \int p_{l}(x_{i1}, x_{j1}, y) p_{l^{\prime}}(x_{i2}, x_{j2}, y)
    \eta(x_{i1}, x_{j2}, y) \, dx_{i1} \, dx_{j1} \, dx_{i2} \,
    dx_{j2} \, dy.
  \end{align*}
  Let us first compute $\E[{\hat{\theta}_{n}^{1}}]$. %
  \begin{align*}
    \E[{\hat{\theta}_{n}^{1}}] = & \ \sum_{l\in M} \int p_{l}(x_{i1},
    x_{j1}, y) f(x_{i1}, x_{j1}, y) \, dx_{i1} \,
    dx_{j1} \, dy \\
    & \int p_{l}(x_{i1}, x_{j1}, y) \psi(x_{i1}, x_{j1}, x_{i2},
    x_{j2}, y) \\
    & \qquad f(x_{i2}, x_{j2}, y) \, dx_{i1}
    \, dx_{j1} \, dx_{i2} \, dx_{j2} \, dy \\
    = & \ \sum_{l\in M} a_{l}\int p_{l}(x_{i1}, x_{j1}, y)
    \psi(x_{i1}, x_{j1}, x_{i2}, x_{j2}, y) \\
    & \qquad f(x_{i2}, x_{j2}, y) \, dx_{i1}
    \, dx_{j1} \, dx_{i2} \, dx_{j2} \, dy \\
    = & \ \int \left(\sum_{l\in M} a_{l} p_{l}(x_{i2}, x_{j2}, y)
    \right) \psi(x_{i1}, x_{j1}, x_{i2}, x_{j2}, y) \\
    & \qquad f(x_{i2}, x_{j2}, y) \, dx_{i1}
    \, dx_{j1} \, dx_{i2} \, dx_{j2} \, dy \\
    = & \ \int S_{M}f(x_{i1}, x_{j1}, y) f(x_{i2}, x_{j2}, y) \\
    & \qquad\eta(x_{i1}, x_{j2}, y) \, dx_{i1}\, dx_{j1} \, dx_{i2}\,
    dx_{j2}\, dy \\
    & + \int S_{M}f(x_{i2}, x_{j2}, y) f(x_{i1}, x_{j1}, y) \\
    & \qquad \eta(x_{i1}, x_{j2}, y) \, dx_{i1}\, dx_{j1}\, dx_{i2}\,
    dx_{j2}\, dy
  \end{align*}
  Now for $\hat{\theta}_{n}^{2}$, we get
  \begin{align*}
    \E[{\hat{\theta}_{n}^{2}}] = & \ \sum_{l, l^{\prime}\in M} \int
    p_{l}(x_{i}, x_{j}, y)
    f(x_{i}, x_{j}, y) \, dx_{i} \, dx_{j} \, dy \\
    & \int p_{l^{\prime}}(x_{i}, x_{j}, y) f(x_{i}, x_{j}, y) \,
    dx_{i}\, dx_{j}\, dy \\
    & \int p_{l}(x_{i1}, x_{j1}, y) p_{l^{\prime}}(x_{i2}, x_{j2}, y)
    \\
    & \qquad \eta(x_{i1}, x_{j2}, y) \, dx_{i1}\, dx_{j1}\, dx_{i2}\,
    dx_{j2}\,
    dy \\
    = & \ \sum_{l, l^{\prime}\in M} a_{l} a_{l^{\prime}} \int
    p_{l}(x_{i1}, x_{j1}, y) p_{l^{\prime}}(x_{i2}, x_{j2}, y) \\
    & \qquad \eta(x_{i1}, x_{j2}, y) \, dx_{i1}\, dx_{j1}\, dx_{i2}\,
    dx_{j2}\, dy \\
    = & \int \left( \sum_{l\in M} a_{l} p_{l}(x_{i1}, x_{j1}, y)
    \right) \left(\sum_{l^{\prime}\in M} a_{l^{\prime}}
      p_{l^{\prime}}(x_{i2}, x_{j2}, y) \right) \\
    & \qquad \eta(x_{i1}, x_{j2}, y) \, dx_{i1}\, dx_{j1}\, dx_{i2}\,
    dx_{j2}\, dy \\
    = & \ \int S_{M}f(x_{i1}, x_{j1}, y) S_{M}f(x_{i2}, x_{j2}, y) \\
    & \qquad \eta(x_{i1}, x_{j2}, y) \, dx_{i1}\, dx_{j1}\, dx_{i2}\,
    dx_{j2}\, dy.
  \end{align*}
  Arranging these terms and using
  \[
  \B({\hat{\theta}_{n}}){} = \E[{\hat{\theta}_{n}}] - \theta =
  \E[{\hat{\theta}_{n}^{1}}] - \E[{\hat{\theta}_{n}^{2}}] - \theta
  \]
  we obtain the desire bias. %
\end{proof}

\begin{lem}[Bound of $\V({U_{n}K})$]
  \label{lem:Var-UK}
  Under the assumptions of Theorem \ref{thm:taylor:quad-estimator}, we
  have
  \[
  \V({U_{n}K}) \leq
  \frac{20}{n(n-1)}\Vert{\eta}\Vert_{\infty}^{2}\Vert
  f\Vert_{\infty}^{2}\Delta^2(m+1)
  \]
\end{lem}
\begin{proof}[\textbf{Proof of Lemma \ref{lem:Var-UK}}]
  Note that $U_{n}K$ is centered because $Q$ and $R$ are centered and
  $(X_{ik}, X_{jk}, Y_{k})$, $k=1, \ldots, n$ is an independent
  sample. %
  So $\V({U_{n}K})$ is equal to
  \begin{align*}
    \E[{U_{n}K}]^{2} = & \ \mathbb{E}
    \Biggl(\frac{1}{\left(n(n-1)\right)^{2}} \sum_{k_{1}\neq
      k_{1}^{\prime}=1}^{n} \sum_{k_{2}\neq k_{2}^{\prime}=1}^{n} K
    \bigl(X_{ik_{1}}, X_{jk_{1}}, Y_{k_{1}}, X_{ik_{1}^{\prime}},
    X_{jk_{1}^{\prime}}, Y_{k_{1}^{\prime}}
    \bigr) \\
    & \qquad K\bigl(X_{ik_{2}}, X_{jk_{2}}, Y_{k_{2}},
    X_{ik_{2}^{\prime}}, X_{jk_{2}^{\prime}},
    Y_{k_{2}^{\prime}}\bigr) \Biggr) \\
    = & \frac{1}{n(n-1)} \mathbb{E} \Biggl(K^{2} \bigl(X_{i1}, X_{j1},
    Y_{1}, X_{i2}, X_{j2}, Y_{2}\bigr) \\
    & \qquad + K \bigl(X_{i1}, X_{j1}, Y_{1}, X_{i2}, X_{j2},
    Y_{2}\bigr) K \bigl(X_{i2}, X_{j2}, Y_{2}, X_{i1}, X_{j1},
    Y_{1}\bigr)\Biggr)
  \end{align*}
  By the Cauchy-Schwarz inequality, we get
  \[
  \V({U_{n}K}) \leq \frac{2}{n(n-1)} \E\bigl[{K^{2} \left(X_{i1},
      X_{j1}, Y_{1}, X_{i2}, X_{j2}, Y_{2}\right)}\bigr].
  \]
  Moreover, using the fact that $2\vert{\E[{XY}]}\vert
  \leq\E[{X^{2}}]+\E[{Y^{2}}]$, we obtain
  \begin{align*}
    & \hspace{-3em} \E[{K^{2}\left(X_{i1}, X_{j1},
        Y_{1}, X_{i2}, X_{j2}, Y_{2}\right)}] \\
    \leq & \ 2\Biggl[ \E[{\bigl(Q^{\top}(X_{i1}, X_{j1}, Y_{1})
      R(X_{i2}, X_{j2}, Y_{2})\bigr)^{2}}] \\
    & \qquad + \E[{\bigl( Q^{\top}(X_{i1}, X_{j1}, Y_{1}) C Q(X_{i2},
      X_{j2}, Y_{2})\bigr)^{2}}] \Biggr].
  \end{align*}
  We will bound these two terms. The first one is
  \begin{align*}
    & \E\Bigl[{\Bigl( Q^{\top}(X_{i1}, X_{j1}, Y_{1})
      R(X_{i2}, X_{j2}, Y_{2})\Bigr)^{2}}\Bigr] \\
    = & \ \sum_{l, l^{\prime}\in M} \left(\int p_{l}(x_{i}, x_{j}, y)
      p_{l^{\prime}}(x_{i}, x_{j}, y) f(x_{i}, x_{j}, y) \, dx_{i}\,
      dx_{j}\, dy -
      a_{l} a_{l^{\prime}} \right) \\
    & \biggl(\int p_{l}(x_{i2}, x_{j2}, y) p_{l^{\prime}}(x_{i3},
    x_{j3}, y) \psi(x_{i1}, x_{j1}, x_{i2},
    x_{j2}, y) \\
    & \psi(x_{i1}, x_{j1}, x_{i3}, x_{j3}, y) f(x_{i1}, x_{j1}, y) \,
    dx_{i1} \, dx_{j1} \, dx_{i2} \, dx_{j2} \, dx_{i3} \, dx_{j3} \,
    dy -  b_{l} b_{l^{\prime}} \biggr) \\
    = & \ W_{1} - W_{2} - W_{3} + W_{4}
  \end{align*}
  where
  \begin{align*}
    W_{1} & = \int \sum_{l, l^{\prime}\in M} p_{l}(x_{i1}, x_{j1}, y)
    p_{l^{\prime}}(x_{i1}, x_{j1}, y) p_{l}(x_{i2}, x_{j2},
    y^{\prime}) p_{l^{\prime}}(x_{i3}, x_{j3}, y^{\prime}) \\
    &\qquad \psi(x_{i4}, x_{j4}, x_{i2}, x_{j2}, y^{\prime})
    \psi(x_{i4}, x_{j4}, x_{i3}, x_{j3}, y^{\prime}) \\
    & \qquad f(x_{i1}, x_{j1}, y) f(x_{i4}, x_{j4}, y^{\prime}) \,
    dx_{i1} \, dx_{j1} \, dx_{i2} \, dx_{j2} \, dx_{i3} \, dx_{j3} \,
    dx_{i4} \,
    dx_{j4} \, dy \, dy^{\prime} \\
    W_{2} & = \int \sum_{l, l^{\prime}\in M} b_{l} b_{l^{\prime}}
    p_{l}(x_{i1}, x_{j1}, y) p_{l^{\prime}}(x_{i1}, x_{j1}, y)
    f(x_{i1}, x_{j1}, y) \, dx_{i1} \, dx_{j1} dy \\
    W_{3} & = \int \sum_{l, l^{\prime}\in M} a_{l} a_{l^{\prime}}
    p_{l}(x_{i2}, x_{j2}, y^{\prime}) p_{l^{\prime}}(x_{i3}, x_{j3},
    y^{\prime}) \\
    & \qquad \psi(x_{i4}, x_{j4}, x_{i2}, x_{j2}, y^{\prime})
    \psi(x_{i4},
    x_{j4}, x_{i3}, x_{j3}, y^{\prime}) \\
    & \qquad f(x_{i4}, x_{j4}, y^{\prime}) \, dx_{i2} \, dx_{j2} \,
    dx_{i3} \, dx_{j3}
    \, dx_{i4} \, dx_{j4} \, dy^{\prime} \\
    W_{4} & = \sum_{l, l^{\prime}\in M} a_{l} a_{l^{\prime}} b_{l}
    b_{l^{\prime}}.
  \end{align*}

  $W_{2}$ and $W_{3}$ are positive, hence
  \begin{equation*}
    \E\Bigl[{\Bigl( 2 Q^{\top}(X_{i1}, X_{j1}, Y_{1}) R(X_{i2}, X_{j2},
      Y_{2}) \Bigl)^{2}}\Bigl] \leq W_{1} + W_{4}.
  \end{equation*}
  \begin{align*}
    W_{1} = & \ \int \sum_{l, l^{\prime}\in M} p_{l}(x_{i1}, x_{j1},
    y) p_{l^{\prime}}(x_{i1}, x_{j1}, y) \\
    & \qquad \left( \int p_{l}(x_{i2}, x_{j2}, y^{\prime})
      \psi(x_{i4}, x_{j4}, x_{i2}, x_{j2}, y^{\prime}) \, dx_{i2} \,
      dx_{j2} \right) \\
    & \qquad \left( \int p_{l^{\prime}}(x_{i3}, x_{j3}, y^{\prime})
      \psi(x_{i4}, x_{j4}, x_{i3}, x_{j3}, y^{\prime}) \, dx_{i3}\,
      dx_{j3} \right) \\
    & \qquad f(x_{i1}, x_{j1}, y) f(x_{i4}, x_{j4}, y^{\prime}) \,
    dx_{i1} \, dx_{j1}
    \, dx_{i4} \, dx_{j4} \, dy \, dy^{\prime}\\
    \leq & \ \Vert f\Vert_{\infty}^{2} \sum_{l, l^{\prime}\in M} \int
    p_{l}(x_{i1}, x_{j1}, y) p_{l^{\prime}}(x_{i1}, x_{j1}, y) \,
    dx_{i1} \, dx_{j1} dy\\
    & \qquad \int \left(\int p_{l}(x_{i2}, x_{j2}, y^{\prime})
      \psi(x_{i4}, x_{j4}, x_{i2}, x_{j2}, y^{\prime}) \, dx_{i2} \,
      dx_{j2} \right) \\
    & \qquad \left( \int p_{l^{\prime}}(x_{i3}, x_{j3}, y^{\prime})
      \psi(x_{i4}, x_{j4}, x_{i3}, x_{j3}, y^{\prime}) \, dx_{i3}\,
      dx_{j3}\right) \\
    & \hspace{15em} \, dx_{i2} \, dx_{j2} \, dx_{i4}
    \, dx_{j4} \, dy^{\prime} \\
  \end{align*}
    
  Since $p_{l}$'s are orhonormal we have
  \begin{multline*}
    W_{1}\leq \Vert f\Vert_{\infty}^{2}\sum_{l\in M} \int \biggl(
    \int  p_{l}(x_{i2}, x_{j2}, y^{\prime}) \\
    \psi(x_{i4}, x_{j4}, x_{i2}, x_{j2}, y^{\prime}) \, dx_{i2} \,
    dx_{j2} \biggr)^{2} \, dx_{i4} \, dx_{j4} \, dy^{\prime}.
  \end{multline*}

  Moreover by the Cauchy-Schwarz inequality and
  $\Vert{\psi}\Vert_{\infty} \leq 2 \Vert{\eta}\Vert_{\infty}$
  \begin{align*}
    & \hspace{-3em} \left(\int p_{l}(x_{i2}, x_{j2}, y^{\prime})
      \psi(x_{i4}, x_{j4}, x_{i2}, x_{j2}, y^{\prime}) \, dx_{i2} \,
      dx_{j2}
    \right)^{2} \\
    & \leq \int p_{l}(x_{i2}, x_{j2}, y^{\prime})^{2} \, dx_{i2} \,
    dx_{j2} \int \psi(x_{i4}, x_{j4}, x_{i2}, x_{j2}, y^{\prime})^{2}
    \,
    dx_{i2} \, dx_{j2} \\
    & \leq \Vert{\psi}\Vert_{\infty}^{2} \Delta\int
    p_{l}(x_{i2}, x_{j2}, y^{\prime})^{2} \, dx_{i2} \, dx_{j2}\\
    & \leq 4\Vert{\eta}\Vert_{\infty}^{2}\Delta \int p_{l}(x_{i2},
    x_{j2}, y^{\prime})^{2}\, dx_{i2}\, dx_{j2},
  \end{align*}
  and then
  \begin{align*}
    & \int\Bigl(\int p_{l}(x_{i2}, x_{j2}, y^{\prime}) \psi(x_{i4},
    x_{j4}, x_{i2}, x_{j2}, y^{\prime}) \, dx_{i2}
    \, dx_{j2}\Bigr)^{2} \, dx_{i4} \, dx_{j4} \, dy^{\prime}\\
    & \leq4\Vert{\eta}\Vert_{\infty}^{2} \Delta^2\int
    p_{l}(x_{i2}, x_{j2}, y^{\prime})^{2} \, dx_{i2} \, dx_{j2} \, dy^{\prime}\\
    & = 4\Vert{\eta}\Vert_{\infty}^{2}\Delta^2.
  \end{align*}

  Finally,
  \begin{equation*}
    W_{1} \leq 4 \Vert{\eta}\Vert_{\infty}^{2} \Vert
    f\Vert_{\infty}^{2}\Delta^2m.
  \end{equation*}

  For the term $W_{4}$ using the facts that $S_{M}f$ and $S_{M}g$ are
  projection and that $\int f=1$, we have
  \[
  W_{4} = \left(\sum_{l\in M}a_{l}b_{l}\right)^{2}\leq\sum_{l\in
    M}a_{l}^{2} \sum_{l\in M}b_{l}^{2} \leq\Vert f\Vert_{2}^{2}\Vert
  g\Vert_{2}^{2}\leq\Vert f\Vert_{\infty}\Vert g\Vert_{2}^{2}.
  \]
  By the Cauchy-Schwartz inequality we have $\Vert
  g\Vert_{2}^{2}\leq4\Vert{\eta}\Vert_{\infty}^{2}\Vert
  f\Vert_{\infty}\Delta^2$ and then
  \[
  W_{4}\leq 4 \Vert{\eta}\Vert_{\infty}^{2} \Vert
  f\Vert_{\infty}^{2}\Delta^2
  \]
  which leads to
  \begin{equation}
    \label{eq:taylor:Var-UK-1term-2QtR}
    \E\Bigl[{ \Bigl( Q^{\top}(X_{i1}, X_{j1}, Y_{1}) R(X_{i2},
      X_{j2}, Y_{2}) \Bigr)^{2}}\Bigl] \leq 4 \Vert{\eta}\Vert_{\infty}^{2}
    \Vert f\Vert_{\infty}^{2} \Delta^2(m+1).
  \end{equation}

  The second term is
  \begin{equation*}
    \E\bigl[\bigl( Q^{\top}(X_{i1}, X_{j1}, Y_{1} C
    Q(X_{i2}, X_{j2}, Y_{2}) \bigr)\bigl] = W_{5} - 2 W_{6} + W_{7}
  \end{equation*}
  where
  \begin{align*}
    W_{5} = & \ \int \sum_{l_{1}, l_{1}^{\prime}} \sum_{l_{2},
      l_{2}^{\prime}} c_{l_{1}l_{1}^{\prime}} c_{l_{2}l_{2}^{\prime}}
    p_{l_{1}}(x_{i1}, x_{j1}, y) p_{l_{2}}(x_{i1}, x_{j1}, y) \\
    & \qquad p_{l_{1}^{\prime}}(x_{i2}, x_{j2}, y^{\prime})
    p_{l_{2}^{\prime}}(x_{i2}, x_{j2}, y^{\prime}) \\
    & \qquad f(x_{i1}, x_{j1}, y) f(x_{i2}, x_{j2}, y^{\prime}) \,
    dx_{i1} \, dx_{j1}
    \, dx_{i2} \, dx_{j2} \, dy^{\prime}\, dy \\
    W_{6} = &\ \int\sum_{l_{1}, l_{1}^{\prime}} \sum_{l_{2},
      l_{2}^{\prime}} c_{l_{1}l_{1}^{\prime}} c_{l_{2}l_{2}^{\prime}}
    a_{l_{1}} a_{l_{2}} p_{l_{1}^{\prime}}(x_{i}, x_{j}, y) \\
    & \qquad p_{l_{2}^{\prime}}(x_{i}, x_{j}, y) \, dx_{i} \, dx_{j}
    \, dy \\
    W_{7} = &\ \sum_{l_{1}, l_{1}^{\prime}} \sum_{l_{2},
      l_{2}^{\prime}} c_{l_{1}l_{1}^{\prime}} c_{l_{2}l_{2}^{\prime}}
    a_{l_{1}} a_{l_{1}^{\prime}} a_{l_{2}} a_{l_{2}^{\prime}}.
  \end{align*}
  Using the previous manipulation, we show that $W_{6}\geq0$. Thus
  \begin{equation*}
    \E\bigl[\bigl( Q^{\top}(X_{i1}, X_{j1}, Y_{1}) C Q(X_{i2}, X_{j2},
    Y_{2}) \bigr)\bigl] \leq W_{5} + W_{7}.
  \end{equation*}

  First, observe that
  \begin{align*}
    W_{5} = &\ \sum_{l_{1}, l_{1}^{\prime}} \sum_{l_{2},
      l_{2}^{\prime}} c_{l_{1}l_{1}^{\prime}} c_{l_{2}l_{2}^{\prime}} \\
    & \qquad \left( \int p_{l_{1}}(x_{i1}, x_{j1}, y)
      p_{l_{2}}(x_{i1}, x_{j1}, y) f(x_{i1}, x_{j1}, y) \, dx_{i1}\,
      dx_{j1}\, dy\right) \\
    & \qquad \left( \int p_{l_{1}^{\prime}}(x_{i2}, x_{j2},
      y^{\prime}) p_{l_{2}^{\prime}}(x_{i2}, x_{j2}, y^{\prime})
      f(x_{i2}, x_{j2}, y^{\prime}) \, dx_{i2}\, dx_{j2}\,
      dy^{\prime}\right) \\
    \leq &\ \Vert f\Vert_{\infty}^{2}\sum_{l_{1}, l_{1}^{\prime}}
    \sum_{l_{2}, l_{2}^{\prime}} c_{l_{1}l_{1}^{\prime}}
    c_{l_{2}l_{2}^{\prime}} \left( \int p_{l_{1}}(x_{i1}, x_{j1}, y)
      p_{l_{2}}(x_{i1}, x_{j1}, y) \, dx_{i1} \, dx_{j1} \, dy\right) \\
    & \qquad \left(\int p_{l_{1}^{\prime}}(x_{i2}, x_{j2}, y^{\prime})
      p_{l_{2}^{\prime}}(x_{i2}, x_{j2}, y^{\prime})
      \, dx_{i2}\, dx_{j2}\, dy^{\prime}\right) \\
    = &\ \Vert f\Vert_{\infty}^{2} \sum_{l, l^{\prime}}
    c_{ll^{\prime}}^{2}
  \end{align*}
  again using the orthonormality of the $p_{l}$'s. %

Therefore, given the decomposition $p_{l}(x_{i}, x_{j}, y) =
  \alpha_{l_{\alpha}}(x_{i}, x_{j}) \beta_{l_{\beta}}(y)$,
  \begin{align*}
    \sum_{l, l^{\prime}}c_{ll^{\prime}}^{2} = & \int \sum_{l_{\beta},
      l_{\beta}^{\prime}}
    \beta_{l_{\beta}}(y)\beta_{l_{\beta}^{\prime}}(y)
    \beta_{l_{\beta}}(y^{\prime})
    \beta_{l_{\beta}^{\prime}}(y^{\prime}) \\
    & \sum_{l_{\alpha}, l_{\alpha}^{\prime}}
    \left(\int\alpha_{l_{\alpha}}(x_{i1}, x_{j1})
      \alpha_{l_{\alpha}^{\prime}}(x_{i2}, x_{j2}) \eta(x_{i1},
      x_{j2}, y) \, dx_{i1} \, dx_{j1} \, dx_{i2} \,
      dx_{j2} \right) \\
    & \left( \int \alpha_{l_{\alpha}}(x_{i3}, x_{j3})
      \alpha_{l_{\alpha}^{\prime}}(x_{i4}, x_{j4}) \eta(x_{i3},
      x_{j4}, y^{\prime}) \, dx_{i3} \, dx_{j3} \, dx_{i4}\, dx_{j4}
    \right) dy dy^{\prime}
  \end{align*}
  But
  \begin{align*}
    & \sum_{l_{\alpha}, l_{\alpha}^{\prime}}
    \left(\int\alpha_{l_{\alpha}}(x_{i1}, x_{j1})
      \alpha_{l_{\alpha}^{\prime}}(x_{i2}, x_{j2}) \eta(x_{i1},
      x_{j2}, y) \, dx_{i1} \, dx_{j1} \, dx_{i2}\,
      dx_{j2}\right) \\
    &\qquad \left( \int \alpha_{l_{\alpha}}(x_{i3}, x_{j3})
      \alpha_{l_{\alpha}^{\prime}}(x_{i4}, x_{j4}) \eta(x_{i3},
      x_{j4}, y^{\prime}) \, dx_{i3}\, dx_{j3}\,
      dx_{i4}\, dx_{j4}\right) \\
    = & \ \sum_{l_{\alpha}, l_{\alpha}^{\prime}} \int
    \alpha_{l_{\alpha}}(x_{i1}, x_{j1})
    \alpha_{l_{\alpha}^{\prime}}(x_{i2}, x_{j2}) \eta(x_{i1}, x_{j2},
    y)
    \alpha_{l_{\alpha}}(x_{i3}, x_{j3}) \\
    & \qquad \alpha_{l_{\alpha}^{\prime}}(x_{i4}, x_{j4}) \eta(x_{i3},
    x_{j4}, y^{\prime}) \, dx_{i1} \, dx_{j1} \, dx_{i2} \, dx_{j2}
    \, dx_{i3} \, dx_{j3} \, dx_{i4} \, dx_{j4} \\
    = & \ \int \sum_{l_{\alpha}} \left(\int\alpha_{l_{\alpha}}(x_{i1},
      x_{j1}) \eta(x_{i1}, x_{j2}, y) \, dx_{i1} \, dx_{j1} \right)
    \alpha_{l_{\alpha}}(x_{i3},
    x_{j3}) \\
    & \qquad \sum_{l_{\alpha}^{\prime}} \left( \int
      \alpha_{l_{\alpha}^{\prime}}(x_{i4}, x_{j4}) \eta(x_{i3},
      x_{j4}, y^{\prime}) \, dx_{i4} \, dx_{j4}\right) \\
    & \qquad \alpha_{l_{\alpha}^{\prime}}(x_{i2}, x_{j2}) \, dx_{i2}
    \,
    dx_{j2} \, dx_{i3} \, dx_{j3} \\
    \leq & \ \int \eta(x_{i3}, x_{j3}, x_{i2}, x_{j2}, y) \eta(x_{i3},
    x_{j2}, y^{\prime}) \, dx_{i2}\, dx_{j2}\, dx_{i3}\,
    dx_{j3} \\
    \leq & = \Delta^2\Vert{\eta}\Vert_{\infty}^{2}
  \end{align*}
  using the orthonormality of the basis $\alpha_{l_{\alpha}}$. Then we
  get
  \begin{align*}
    \sum_{l, l^{\prime}} c_{ll^{\prime}}^{2} & \leq \Delta^2
    \Vert{\eta} \Vert_{\infty}^{2} \left( \int \sum_{l_{\beta},
        l_{\beta}^{\prime}}
      \beta_{l_{\beta}}(y)\beta_{l_{\beta}^{\prime}}(y)
      \beta_{l_{\beta}}(y^{\prime})
      \beta_{l_{\beta}^{\prime}}(y^{\prime}) \, dy \, dy^{\prime} \right)\\
    & = \Delta^2 \Vert{\eta}\Vert_{\infty}^{2}\sum_{l_{\beta},
      l_{\beta}^{\prime}} \left(\int\beta_{l_{\beta}}(y)
      \beta_{l_{\beta}^{\prime}}(y) \, dy
    \right)^{2} \\
    & \leq \Delta^2\Vert{\eta}\Vert_{\infty}^{2} \sum_{l_{\beta}}
    \left( \int
      \beta_{l_{\beta}}^{2}(y)\, dy\right)^{2}\\
    & \leq \Delta^2\Vert{\eta}\Vert_{\infty}^{2}m
  \end{align*}
  since the $\beta_{l_{\beta}}$ are orthonormal. %
  Finally
  \[
  W_{5}\leq \Vert f\Vert_{\infty}^{2}
  \Vert{\eta}\Vert_{\infty}^{2}\Delta^2m.
  \]
  Now for $W_{7}$ we first will bound,
  \begin{align*}
    & \left \vert{ \sum_{l, l^{\prime}} c_{ll^{\prime}} a_{l}
        a_{l^{\prime}}} \right \vert \\
    = & \ \Bigg \vert \int \sum_{l, l^{\prime}\in M} a_{l}
    a_{l^{\prime}} p_{l_{2}}(x_{i1},
    x_{j1}, y) p_{l_{1}^{\prime}}(x_{i2}, x_{j2}, y) \\
    & \qquad \eta(x_{i1}, x_{j2}, y) \, dx_{i1} \, dx_{j1}\,
    dx_{i2}\, dx_{j2}\, dy \Bigg \vert \\
    \leq & \ \int \Bigl \vert S_{M}(x_{i1}, x_{j1}, y) S_{M}(x_{i2},
    x_{j2}, y) \\
    &\qquad \eta(x_{i1}, x_{j2}, y) \Bigr\vert
    \, dx_{i1} \, dx_{j1} \, dx_{i2}\, dx_{j2}\, dy \\
    \leq & \ \Vert{\eta}\Vert_{\infty} \int
    \left(\int\vert{S_{M}(x_{i1}, x_{j1}, y) S_{M}(x_{i2}, x_{j2},
        y)}\vert \, dy \right) \, dx_{i1} \, dx_{j1} \, dx_{i2}\,
    dx_{j2}.
  \end{align*}

  Taking squares in both sides and using the Cauchy-Schwartz
  inequality twice, we get
  \begin{align*}
    & \left( \sum_{l, l^{\prime}} c_{ll^{\prime}} a_{l}
      a_{l^{\prime}} \right)^{2} \\
    = & \ \Vert{\eta}\Vert_{\infty}^{2} \left(\int
      \left(\int\vert{S_{M}(x_{i1}, x_{j1}, y) S_{M}(x_{i2}, x_{j2},
          y)}\vert\, dy\right)
      \, dx_{i1}\, dx_{j1}\, dx_{i2}\, dx_{j2}\right)^{2} \\
    \leq & \ \Vert{\eta}\Vert_{\infty}^{2} \Delta^2 \int
    \left(\int\vert{ S_{M}(x_{i1}, x_{j1}, y) S_{M}(x_{i2}, x_{j2},
        y)}\vert\, dy\right)^{2}
    \, dx_{i1}\, dx_{j1}\, dx_{i2}\, dx_{j2} \\
    \leq & \ \Vert{\eta}\Vert_{\infty}^{2} \Delta^2 \int \left(\int
      S_{M}(x_{i1}, x_{j1}, y)^{2}\, dy\right) \\
    & \qquad \left( \int S_{M}(x_{i2}, x_{j2}, y^{\prime})^{2}\,
      dy^{\prime} \right)
    \, dx_{i1}\, dx_{j1}\, dx_{i2}\, dx_{j2} \\
    = & \ \Vert{\eta}\Vert_{\infty}^{2} \Delta^2\int S_{M}(x_{i1},
    x_{j1}, y)^{2} \\
    & \qquad S_{M}(x_{i1}, x_{j1}, y^{\prime})^{2}
    \, dx_{i1}\, dx_{j1}\, dx_{i2}dx_{j2}\, dy\, dy^{\prime} \\
    = & \ \Vert{\eta}\Vert_{\infty}^{2} \Delta^2\left( \int
      S_{M}(x_{i},
      x_{j}, y)^{2} \, dx_{i}\, dx_{j}\, dy\right) \\
    \leq & \ \Vert{\eta}\Vert_{\infty}^{2}\Delta^2\Vert
    f\Vert_{\infty}^{2}.
  \end{align*}
  Finally,
  \begin{equation}
    \label{eq:taylor:Var-UK-2term-QtCQ}
    \E[{ \bigl( Q^{\top}(X_{i1}, X_{j1}, Y_{1}) C
      Q(X_{i2}, X_{j2}, Y_{2}) \bigr)^{2}}] \leq
    \Vert{\eta}\Vert_{\infty}^{2} \Vert f\Vert_{\infty}^{2} \Delta^2(m+1).
  \end{equation}
  Collecting \eqref{eq:taylor:Var-UK-1term-2QtR} and
  \eqref{eq:taylor:Var-UK-2term-QtCQ}, we obtain
  \[
  \V({U_{n}K})\leq\frac{20}{n(n-1)} \Vert{\eta}\Vert_{\infty}^{2}
  \Vert f\Vert_{\infty}^{2} \Delta^2(m+1)
  \]
  which concludes the proof of Lemma \ref{lem:Var-UK}.
\end{proof}

\begin{lem}[Bound for $\V({P_{n}L})$]
  \label{lem:Var-PL}
  Under the assumptions of Theorem \ref{thm:taylor:quad-estimator}, we
  have
  \[
  \V({P_{n}L})\leq\frac{12}{n}\Vert{\eta}\Vert_{\infty}^{2}\Vert
  f\Vert_{\infty}^{2}\Delta^2.
  \]
\end{lem}
\begin{proof}[\textbf{Proof of Lemma \ref{lem:Var-PL}}]
  First note that given the independence of $\bigl( X_{ik}, X_{jk},
  Y_k\bigr)$ for $k=1, \ldots, n$ we have
  \[
  \V({P_{n}L})=\frac{1}{n}\V({L\bigl(X_{i1}, X_{j1}, Y_1\bigr)})
  \]
  we can write $L\bigl(X_{i1}, X_{j1}, Y_1\bigr)$ as
  \begin{align*}
    & A^{\top} R \left(X_{i1}, X_{j1}, Y_1\right) + B^{\top} Q
    \left(X_{i1}, X_{j1}, Y_1\right) - 2 A^{\top} C Q \left(
      X_{i1}, X_{j1}, Y_1 \right)\\
    = & \ \sum_{l\in M} a_{l} \left(\int p_{l}(x_{i}, x_{j}, Y_1)
      \psi(x_{i}, x_{j}, X_{i1}, X_{j1}, Y_1) \, dx_{i} \, dx_{j}-b_{l}\right)\\
    & \qquad + \sum_{l\in M} b_{l} \left( p_{l}(X_{i1}, X_{j1}, Y_1) -
      a_{l} \right) - 2 \sum_{l, l^{\prime}\in M}c_{ll^{\prime}}
    a_{l^{\prime}} \left( p_{l}(X_{i1}, X_{j1}, Y_1) - a_{l} \right) \\
    = & \ \int \sum_{l\in M} a_{l} p_{l}(x_{i}, x_{j}, Y_1)
    \psi(x_{i}, x_{j}, X_{i1}, X_{j1}, Y_1) \, dx_{i} \, dx_{j} \\
    & \qquad + \sum_{l\in M} b_{l} p_{l}(X_{i1}, X_{j1}, Y_1) \\
    & \qquad - 2 \sum_{l, l^{\prime}\in M} c_{ll^{\prime}}
    a_{l^{\prime}} p_{l}(X_{i1}, X_{j1}, Y_1) - 2 A^{t} B - 2 A^{t} C A.\\
    = & \ \int S_{M}f(x_{i}, x_{j}, Y_1) \psi(x_{i}, x_{j}, X_{i1},
    X_{j1}, Y_1) \, dx_{i} \, dx_{j} + S_{M}g(X_{i1}, X_{j1}, Y_1)\\
    & \qquad - 2 \sum_{l, l^{\prime}\in M} c_{ll^{\prime}}
    a_{l^{\prime}} p_{l}(X_{i1}, X_{j1}, Y_1) - 2 A^{\top} B - 2
    A^{\top} C A.
  \end{align*}
  Let $h(x_{i}, x_{j}, y) = \int S_{M}f(x_{i2}, x_{j2}, y) \psi(x_{i},
  x_{j}, x_{i2}, x_{j2}, y) \, dx_{i2} \, dx_{j2}$, we have
  \begin{align*}
    & S_{M}h(x_{i}, x_{j}, y) \\
    = & \ \sum_{l\in M} \left(\int h(x_{i2}, x_{j2}, y) p_{l}(x_{i2},
      x_{j2}, y) \, dx_{i2}\, dx_{j2}\, dy \right)
    p_{l}(x_{i}, x_{j}, y) \\
    = & \ \sum_{l\in M} \bigg(\int S_{M}f(x_{i3}, x_{j3}, y)
    \psi(x_{i2}, x_{j2}, x_{i3}, x_{j3}, y) \\
    & \qquad p_{l}(x_{i2}, x_{j2}, y) \, dx_{i2}\, dx_{j2}\, dx_{i3}\,
    dx_{j3}\, dy \bigg)
    p_{l}(x_{i}, x_{j}, y) \\
    = & \ \sum_{l, l^{\prime}\in M} \bigg(\int a_{l^{\prime}}
    p_{l^{\prime}}(x_{i3}, x_{j3}, y) \psi(x_{i2}, x_{j2}, x_{i3},
    x_{j3}, y) \\
    & \qquad p_{l}(x_{i2}, x_{j2}, y) \, dx_{i2}\, dx_{j2}\, dx_{i3}\,
    dx_{j3}\, dy \bigg) p_{l}(x_{i}, x_{j}, y) \\
    = & \ 2 \sum_{l, l^{\prime}\in M} \bigg(\int a_{l^{\prime}}
    p_{l^{\prime}}(x_{i3}, x_{j3}, y) \eta(x_{i2}, x_{j3}, y) \\
    & \qquad p_{l}(x_{i2}, x_{j2}, y) \, dx_{i2} \, dx_{j2} \, dx_{i3}
    \,
    dx_{j3} \, dy \bigg) p_{l}(x_{i}, x_{j}, y) \\
    = & \ 2 \sum_{l, l^{\prime}\in M} a_{l^{\prime}} c_{ll^{\prime}}
    p_{l}(x_{i}, x_{j}, y)
  \end{align*}
  and we can write
  \begin{align*}
    L \bigl(X_{i1}, X_{j1}, Y_1\bigr) & = h \bigl( X_{i1}, X_{j1},
    Y_1\bigr) + S_{M}g\bigl(X_{i1}, X_{j1}, Y_1 \bigr) \\
    & -S_{M}h\bigl(X_{i1}, X_{j1}, Y_1\bigr) - 2 A^{\top} B - 2
    A^{\top} C A.
  \end{align*}
  Thus,
  \begin{align*}
    & \V({L\bigl(X_{i1}, X_{j1}, Y_1\bigr)}) \\
    & = \V({h\bigl(X_{i1}, X_{j1}, Y_1\bigr) + S_{M}g\bigl(X_{i1},
      X_{j1}, Y_1\bigr) + S_{M}h\bigl(X_{i1}, X_{j1}, Y_1 \bigr)}) \\
    & \leq \E[{\bigl(h\bigl(X_{i1}, X_{j1}, Y_1\bigr) +
      S_{M}g\bigl(X_{i1}, X_{j1}, Y_1\bigr) + S_{M}h\bigl(X_{i1},
      X_{j1}, Y_1\bigr) \bigr)^{2}}] \\
    & \leq \E[{\bigl(h\bigl(X_{i1}, X_{j1},
      Y_1\bigr)\bigr)^{2}+\bigl(S_{M}g\bigl(X_{i1}, X_{j1},
      Y_1\bigr)\bigr)^{2}+\bigl(S_{M}h\bigl(X_{i1}, X_{j1},
      Y_1\bigr)\bigr)^{2}}].
  \end{align*}
  Each of these terms can be bounded
  \begin{align*}
    & \hspace{-5ex} \E[{\bigl(h\bigl(X_{i1}, X_{j1}, Y_1\bigr)\bigr)^{2}}] \\
    = & \ \int\left(\int S_{M}f(x_{i2}, x_{j2}, y) \psi(x_{i1}x_{j2},
      x_{i2}, x_{j2}, y) \, dx_{i2}\,
      dx_{j2}\right)^{2} \\
    & \qquad f(x_{i1}, x_{j1}, y)\, dx_{i1}\, dx_{j1}\, dy \\
    \leq & \ \Delta \int S_{M}f(x_{i2}, x_{j2}, y)^{2}
    \psi(x_{i1}x_{j2}, x_{i2}, x_{j2}, y)^{2} \\
    & \qquad f(x_{i1}, x_{j1}, y)
    \, dx_{i1}\, dx_{j1}\, dx_{i2}\, dx_{j2}\, dy \\
    \leq & \ 4 \Delta^2\Vert
    f\Vert_{\infty}\Vert{\eta}\Vert_{\infty}^{2}\int
    S_{M}f(x_{i}, x_{j}, y)^{2} \, dx_{i}\, dx_{j}\, dy \\
    = & \ 4\Delta^2\Vert
    f\Vert_{\infty}\Vert{\eta}\Vert_{\infty}^{2}\Vert{S_{M}f}\Vert_{2}^{2}
    \\
    \leq & \ 4\Delta^2\Vert
    f\Vert_{\infty}\Vert{\eta}\Vert_{\infty}^{2}\Vert f\Vert_{2}^{2} \\
    \leq & \ 4\Delta^2\Vert
    f\Vert_{\infty}^{2}\Vert{\eta}\Vert_{\infty}^{2}
  \end{align*}
  and similar calculations are valid for the others two terms,
  \begin{align*}
    \E[{\bigl( S_{M}g\bigl(X_{i1}, X_{j1}, Y_1\bigr)\bigr)^{2}}] &
    \leq \Vert f\Vert_{\infty}\Vert{S_{M}g}\Vert_{2}^{2} \leq \Vert
    f\Vert_{\infty}
    \Vert g\Vert_{2}^{2}\leq 4 \Delta^2\Vert f\Vert_{\infty}^{2}\Vert{\eta}\Vert_{\infty}^{2}\\
    \E[{\bigl(S_{M}h\bigl(X_{i1}, X_{j1}, Y_1\bigr)\bigr)^{2}}] & \leq
    \Vert f\Vert_{\infty}\Vert{S_{M}h}\Vert_{2}^{2} \leq \Vert
    f\Vert_{\infty}\Vert h\Vert_{2}^{2} \leq 4\Delta^2 \Vert f
    \Vert_{\infty}^{2} \Vert{\eta} \Vert_{\infty}^{2}.
  \end{align*}
  Finally we get,
  \begin{equation*}
    \V({P_{n}L})\leq\frac{12}{n}\Vert{\eta}\Vert_{\infty}^{2}\Vert
    f\Vert_{\infty}^{2}\Delta^2.
    \qedhere
  \end{equation*}

\end{proof}

\begin{lem}[Computation of $\C({U_{n}K, P_{n}L})$]
  \label{lem:Cov-UK-PL}
  Under the assumptions of Theorem \ref{thm:taylor:quad-estimator}, we
  have
  \[
  \C({U_{n}K, P_{n}L})=0.
  \]

\end{lem}

 \begin{proof}[\textbf{Proof of Lemma \ref{lem:Cov-UK-PL}}]
   Since $U_{n}K$ and $P_{n}L$ are centered, we have
   \begin{align*}
     & \hspace{-5ex} \C({U_{n}K, P_{n}L}) \\
     =\ &  \E[{U_{n}KP_{n}L}] \\
     =\ & \mathbb{E}\left[\frac{1}{n^{2}(n-1)} \sum_{k\neq
         k^{\prime}=1}^{n} K \bigl(X_{ik}, X_{jk}, Y_{k},
       X_{ik^{\prime}}, X_{jk^{\prime}}, Y_{k^{\prime}}\bigr)
       \sum_{k=1}^{n} L \bigl(X_{ik}, X_{jk}, Y_{k}\bigr) \right] \\
     =\ & \frac{1}{n} \E[{ K \bigl(X_{i1}, X_{j1}, Y_{1}, X_{i2},
       X_{j2}, Y_{2}\bigr) \bigl(L\bigl(X_{i1}, X_{j1}, Y_1\bigr)
       + L \bigl(X_{i2}, X_{j2}, Y_2\bigr)\bigr)}] \\
     =\ & \frac{1}{n} \mathbb{E} \bigl[\bigl( Q^{\top}(X_{i1}, X_{j1},
     Y_{1})R(X_{i2}, X_{j2}, Y_{2}) - Q^{\top}(X_{i1}, X_{j1}, Y_{1})
     C  Q(X_{i2}, X_{j2}, Y_{2})\bigr) \\
     & \qquad \bigl(A^{\top} R(X_{i1}, X_{j1}, Y_1) +
     B^{\top}Q(X_{i1}, X_{j1}, Y_1)-2A^{\top} C Q(X_{i1}, X_{j1}, Y_1) \\
     & \qquad + A^{\top} R(X_{i2}, X_{j2}, Y_2) + B^{\top} QX_{i2},
     X_{j2}, Y_2) - 2 A^{\top} C Q(X_{i2}, X_{j2}, Y_2)\bigr)\bigr] \\
     =\ & 0.
   \end{align*}
   Since $K$, $L$, $Q$ and $R$ are centered.
 \end{proof}

 \begin{lem}[Bound of $\B({\hat{\theta}_{n}}){}$]
   \label{lem:Bias-theta-hat-bound}
   Under the assumptions of Theorem \ref{thm:taylor:quad-estimator},
   we have
   \[
   \vert{\B({\hat{\theta}_{n}})} \vert \le \Delta \Vert{\eta}
   \Vert_{\infty} \sup_{l\notin M}\vert{c_{l}}\vert^{2}.
   \]

 \end{lem}

 \begin{proof}[\textbf{Proof of Lemma \ref{lem:Bias-theta-hat-bound}}]
   \begin{align*}
     \vert{\B{\hat{\theta}_{n}}{}}\vert & \leq
     \Vert{\eta}\Vert_{\infty} \int\left(\int \vert{S_{M}f(x_{i1},
         x_{j1}, y) - f(x_{i1}, x_{j1}, y)}\vert
       \, dx_{i1}\, dx_{j1}\right) \\
     & \qquad \left(\int\vert{S_{M}f(x_{i2}, x_{j2}, y) - f(x_{i2},
         x_{j2}, y)}\vert
       \, dx_{i2}\, dx_{j2}\right)\, dy \\
     & = \Vert{\eta}\Vert_{\infty} \int\left(\int\vert{S_{M}f(x_{i},
         x_{j}, y) - f(x_{i}, x_{j}, y)}\vert \, dx_{i}\,
       dx_{j}\right)^{2}\, dy  \\
     & \leq \Delta \Vert{\eta}\Vert_{\infty} \int\left(S_{M}f(x_{i},
       x_{j}, y) - f(x_{i}, x_{j}, y)\right)^{2} \, dx_{i}\, dx_{j}\, dy \\
     & = \Delta \Vert{\eta}\Vert_{\infty}\sum_{l, l^{\prime}\notin M}
     a_{l} a_{l^{\prime}} \int p_{l}(x_{i}, x_{j}, y)
     p_{l^{\prime}}(x_{i}, x_{j}, y) \, dx_{i}\, dx_{j}\, dy \\
     & = \Delta \Vert{\eta}\Vert_{\infty}\sum_{l\notin M}
     \vert{a_{l}}\vert^{2} \leq\Delta \Vert{\eta}\Vert_{\infty}
     \sup_{l\notin M}\vert{c_{l}}\vert^{2}.
   \end{align*}
   We use the H\"older's inequality and the fact that
   $f\in\mathcal{E}$ then
   \begin{equation*}
     \sum_{l\notin M} \vert{a_{l}}\vert^{2} \leq
     \sup_{l\notin M}\vert{c_{l}}\vert^{2}.
     \qedhere
   \end{equation*}

 \end{proof}

 \begin{lem}[Asymptotic variance of $\sqrt{n}\bigl(P_{n}L\bigr)$.]
   \label{lem:Asymp-var-PL}
   Under the assumptions of Theorem \ref{thm:taylor:quad-estimator},
   we have
   \[
   n\, \V({P_{n}L})\to\Lambda(f, \eta)
   \]
   where
   \begin{multline*}
     \Lambda(f, \eta)=\int g(x_{i}, x_{j}, y)^{2} f(x_{i}, x_{j}, y)\,
     dx_{i}\, dx_{j}\, dy  \\
     - \left(\int g(x_{i}, x_{j}, y)f(x_{i}, x_{j}, y) \, dx_{i}\,
       dx_{j}\, dy\right)^{2}.
   \end{multline*}

 \end{lem}

 \begin{proof}[\textbf{Proof of Lemma \ref{lem:Asymp-var-PL}}]
   We proved in Lemma \ref{lem:Var-PL} that
   \begin{align*}
     & \V({L\bigl(X_{i1}, X_{j1}, Y_1\bigr)}) \\
     & = \V(h\bigl(X_{i1}, X_{j1}, Y_1\bigr) \\
     & \qquad + S_{M}g\bigl(X_{i1}, X_{j1}, Y_1\bigr)
     + S_{M}h\bigl(X_{i1}, X_{j1}, Y_1\bigr)) \\
     & = \V({A_{1}+A_{2}+A_{3}}) \\
     & = \sum_{k, l=1}^{3}\C({A_{k}, A_{l}}).
   \end{align*}
   We claim that $\forall k, l\in\{1, 2, 3\}^{2}$, we have
   \begin{align}
     \label{eq:taylor:Bound-Cov-A-A}
     \begin{split}
       & \Bigg\vert \C({A_{k}, A_{l}}) \\
       & \qquad - \epsilon_{kl}\Bigg[\int g(x_{i}, x_{j}, y)^{2}
       f(x_{i}, x_{j}, y) \, dx_{i}\, dx_{j}\, dy \\
       & \qquad - \left(\int g(x_{i}, x_{j}, y) f(x_{i}, x_{j}, y)
         \, dx_{i}\, dx_{j}\, dy\right)^{2}\Bigg] \Bigg\vert \\
       & \leq \lambda \left[\Vert{S_{M}f-f}\Vert_{2} +
         \Vert{S_{M}g-g}\Vert_{2}\right]
     \end{split}
   \end{align}
   where
   \begin{equation*}
     \epsilon_{kl} =
     \begin{cases}
       -1 & \text{if }k=3\text{ or }l=3\text{ and }k\neq l \\
       1 & \text{otherwise}
     \end{cases},
   \end{equation*}
   and $\lambda$ depends only on $\Vert f\Vert_{\infty}$,
   $\Vert{\eta}\Vert_{\infty}$ and $\Delta$. We will do the details
   only for the case $k=l=3$ since the calculations are similar for
   others configurations,
   \begin{multline*}
     \V({A_{3}})=\int S_{M}^{2}h\left(x_{i}, x_{j}, y\right)f(x_{i},
     x_{j}, y)\, dx_{i}\, dx_{j}\, dy \\
     -\left(\int S_{M}h\left(x_{i}, x_{j}, y\right)f(x_{i}, x_{j},
       y)\, dx_{i}\, dx_{j}\, dy\right)^{2}.
   \end{multline*}

   The computation will be done in two steps. We first bound the
   quantity by the Cauchy-Schwartz inequality
   \begin{align*}
     & \Bigg\vert\int S_{M}^{2}h(x_{i}, x_{j}, y) f(x_{i}, x_{j}, y)
     \, dx_{i}\, dx_{j}\, dy \\
     & \qquad - \int g(x_{i}, x_{j}, y)^{2} f(x_{i},
     x_{j}, y)\, dx_{i}\, dx_{j}\, dy\Bigg\vert \\
     \leq\ & \ \int\vert{S_{M}^{2}h(x_{i}, x_{j}, y) f(x_{i}, x_{j},
       y) -
       S_{M}^{2}g(x_{i}, x_{j}, y) f(x_{i}, x_{j}, y)}\vert \, dx_{i}\, dx_{j}\, dy\\
     & \qquad + \int \vert{S_{M}^{2}g(x_{i}, x_{j}, y) f(x_{i}, x_{j},
       y) - g(x_{i}, x_{j}, y)^{2} f(x_{i}, x_{j}, y)}\vert \,
     dx_{i}\,
     dx_{j}\, dy \\
     \leq\ & \ \Vert f\Vert_{\infty}\Vert{S_{M}h + S_{M}g}\Vert_{2}
     \Vert{S_{M}h - S_{M}g}\Vert_{2} + \Vert
     f\Vert_{\infty}\Vert{S_{M}g + g}\Vert_{2} \Vert{S_{M}g -
       g}\Vert_{2}.
   \end{align*}
   Using several times the fact that since $S_{M}$ is a projection,
   $\Vert{S_{M}g}\Vert_{2}\leq\Vert g\Vert_{2}$, the sum is bounded by
   \begin{align*}
     & \Vert f\Vert_{\infty}\Vert{h+g}\Vert_{2}\Vert{h - g}\Vert_{2} +
     2\Vert
     f\Vert_{\infty}\Vert g\Vert_{2}\Vert{S_{M}g - g}\Vert_{2}\\
     & \leq \Vert f\Vert_{\infty}\left(\Vert h\Vert_{2} + \Vert
       g\Vert_{2}\right)\Vert{h - g}\Vert_{2} + 2\Vert
     f\Vert_{\infty}\Vert g\Vert_{2}\Vert{S_{M}g - g}\Vert_{2}.
   \end{align*}
   We saw previously that $\Vert g\Vert_{2}\leq 2 \Delta \Vert f
   \Vert_{\infty}^{1/2} \Vert{\eta} \Vert_{\infty}$ and $\Vert
   h\Vert_{2}\leq 2 \Delta \Vert
   f\Vert_{\infty}^{1/2}\Vert{\eta}\Vert_{\infty}$. %
   The sum is then bound by
   \[
   4\Delta \Vert
   f\Vert_{\infty}^{3/2}\Vert{\eta}\Vert_{\infty}\Vert{h-g}\Vert_{2} +
   4\Delta \Vert
   f\Vert_{\infty}^{3/2}\Vert{\eta}\Vert_{\infty}\Vert{S_{M}g-g}\Vert_{2}.
   \]
   We now have to deal with $\Vert{h-g}\Vert_{2}$:
   \begin{align*}
     & \Vert{h-g}\Vert_{2}^{2}\\
     =\ & \int \Bigg( \int \big(S_{M}f(x_{i2}, x_{j2}, y) - f(x_{i2},
     x_{j2}, y)\big) \\
     & \qquad \psi(x_{i1}, x_{j1}, x_{i2}, x_{j2}, y)
     dx_{i2}dx_{j2}\Bigg)^{2}\, dx_{i1}\, dx_{j1}\, dy \\
     \leq\ & \int\left(\int\left(S_{M}f(x_{i2}, x_{j2}, y) -
         f(x_{i2}, x_{j2}, y)\right)^{2} \, dx_{i2}\, dx_{j2}\right) \\
     & \left(\int \psi^{2}(x_{i1}, x_{j1}, x_{i2}, x_{j2}, y)
       \, dx_{i2}\, dx_{j2}\right)\, dx_{i1}\, dx_{j1}\, dy\\
     \leq\ & 4\Delta^2 \Vert{\eta} \Vert_{\infty}^{2} \Vert{S_{M}f-f}
     \Vert_{2}^{2}.
   \end{align*}
   Finally this first part is bounded by
   \begin{align*}
     & \Bigg\vert\int S_{M}^{2}h\left(x_{i}, x_{j}, y\right) f(x_{i},
     x_{j}, y) \, dx_{i}\, dx_{j}\, dy \\
     & \qquad - \int g(x_{i}, x_{j}, y)^{2}
     f(x_{i}, x_{j}, y) \, dx_{i}\, dx_{j}\, dy\Bigg\vert\\
     \leq\ & 4\Delta \Vert
     f\Vert_{\infty}^{3/2}\Vert{\eta}\Vert_{\infty} \left(2\Delta
       \Vert{\eta}\Vert_{\infty} \Vert{S_{M}f-f}\Vert_{2} +
       \Vert{S_{M}g-g}\Vert_{2} \right).
   \end{align*}
   Following with the second quantity
   \begin{align*}
     & \Biggr\vert\left(\int S_{M}h(x_{i}, x_{j}, y) f(x_{i}, x_{j},
       y) \,
       dx_{i}\, dx_{j}\, dy\right)^{2} \\
     & \qquad - \left(\int g(x_{i}, x_{j}, y) f(x_{i}, x_{j}, y) \,
       dx_{i}\, dx_{j}\,
       dy\right)^{2}\Biggr\vert \\
     =\ & \Biggr|\left(\int \left(S_{M}h(x_{i}, x_{j}, y) -
         g(x_{i}, x_{j}, y) \right) f(x_{i}, x_{j}, y) \, dx_{i}\, dx_{j}\, dy\right)\\
     & \qquad \left(\int\left(S_{M}h(x_{i}, x_{j}, y) + g(x_{i},
         x_{j}, y)\right) f(x_{i}, x_{j}, y) \, dx_{i}\, dx_{j}\,
       dy\right)\Biggr|.
   \end{align*}
   By using the Cauchy-Schwartz inequality, it is bounded by
   \begin{align*}
     & \Vert f\Vert_{2}\Vert{S_{M}h-g}\Vert_{2}\Vert
     f\Vert_{2}\Vert{S_{M}h+g}\Vert_{2} \\
     & \leq \Vert f\Vert_{2}^{2}\left(\Vert h\Vert_{2} + \Vert
       g\Vert_{2}\right)\left(\Vert{S_{M}h-S_{M}g}\Vert_{2} +
       \Vert{S_{M}g-g}\Vert_{2}\right) \\
     & \leq 4\Delta \Vert
     f\Vert_{\infty}^{3/2}\Vert{\eta}\Vert_{\infty}\left(\Vert{h-g}\Vert_{2}
       + \Vert{S_{M}g-g}\Vert_{2}\right) \\
     & \leq 4\Delta \Vert f\Vert_{\infty}^{3/2}
     \Vert{\eta}\Vert_{\infty} \left(2\Delta \Vert{\eta}
       \Vert_{\infty} \Vert{S_{M}f-f} \Vert_{2} +
       \Vert{S_{M}g-g}\Vert_{2}\right)
   \end{align*}
   using the previous calculations. Collecting the two inequalities
   gives \eqref{eq:taylor:Bound-Cov-A-A} for $k=l=3$. Finally, since
   by assumption $\forall t\in\mathbb{L}^{2}(d\mu)$,
   $\Vert{S_{M}t-t}\Vert_{2}\rightarrow0$ when $n\rightarrow\infty$ a
   direct consequence of \eqref{eq:taylor:Bound-Cov-A-A} is
   \begin{align*}
     & \lim_{n\to\infty}\V({L\bigl(X_{i1}, X_{j1}, Y_1\bigr)})\\
     =\ & \int g^{2}(x_{i}, x_{j}, y) f(x_{i}, x_{j}, y) \, dx_{i}\,
     dx_{j}\, dy \\
     & \qquad - \left(\int g(x_{i}, x_{j}, y) f(x_{i}, x_{j}, y)
       \, dx_{i}\, dx_{j}\, dy\right)^{2} \\
     =\ & \Lambda(f, \eta).
   \end{align*}
   We conclude by noticing that $\V({\sqrt{n} \bigl(P_{n}L\bigr)}) =
   \V({L\bigl(X_{i1}, X_{j1}, Y_1\bigr)})$.
 \end{proof}

 \begin{lem}[Asymptotics for $\sqrt{n}(\hat{Q}-Q)$ ]
   \label{lem:Asymp-n-hatQ-Q}
   Under the assumptions of Theorem \ref{thm:taylor:Asymp-norm-Tf-ij},
   we have
   \[
   \lim_{n\to\infty}n\, \E[{\hat{Q}-Q}]^{2}=0.
   \]

 \end{lem}
 \begin{proof}[\textbf{Proof of Lemma \ref{lem:Asymp-n-hatQ-Q}}]
   The bound given in \eqref{eq:taylor:Bound-MSE-quad} states that if
   $\vert{M_{n}}\vert/n\to0$ we have
   \begin{align*}
     & \Biggl| n \, \E[{\bigl(\widehat{Q} - Q\bigr)^{2} \vert\hat{f}}]
     - \Biggl[\int \hat{g}(x_{i}, x_{j}, y)^{2} f(x_{i}, x_{j}, y)
     \, dx_{i}\, dx_{j}\, dy \\
     & \qquad - \left(\int\hat{g}(x_{i}, x_{j}, y) f(x_{i}, x_{j}, y)
       \, dx_{i}\, dx_{j}\,
       dy\right)^{2}\Biggr]\Biggr| \\
     & \leq \gamma\bigl(\Vert f\Vert_{\infty},
     \Vert{\eta}\Vert_{\infty}, \Delta \bigr)
     \left[\frac{\vert{M_{n}}\vert}{n} + \Vert{S_{M}f-f}\Vert_{2} +
       \Vert{S_{M}\hat{g}-\hat{g}}\Vert_{2}\right]
   \end{align*}
   where
   \[
   \hat{g}(x_{i}, x_{j}, y)=\int H_{3}(\hat{f}, x_{i}, x_{j}, x_{i2},
   x_{j2}, y)f(x_{i2}, x_{j2}, y)\, dx_{i2}\, dx_{j2},
   \]
   where we recall that
   \begin{equation*}
     H_{3}(f, x_{i1}, x_{j1}, x_{i2}, x_{j2}, y) =
     H_{2}(f, x_{i1}, x_{j2}, y) + H_{2}(f, x_{i2}, x_{j1}, y)
   \end{equation*}
   with
   \[
   H_{2}(\hat{f}, x_{i1}, x_{j2}, y) = \frac{\bigl(x_{i1} -
     m_{i}(\hat{f}, y)\bigr)\bigl(x_{j2} - m_{j}(\hat{f},
     y)\bigr)}{\int\hat{f}(x_{i}, x_{j}, y) \, dx_{i}\, dx_{j}}.
   \]

   By deconditioning we get
   \begin{align*}
     & \Biggl|n\, \E[{\bigl(\widehat{Q}-Q\bigr)^{2}} ]-
     \mathbb{E}\Biggl[\int\hat{g}(x_{i}, x_{j}, y)^{2} f(x_{i}, x_{j},
     y) \, dx_{i}\, dx_{j}\, dy \\
     & \qquad - \left(\int\hat{g}(x_{i}, x_{j}, y) f(x_{i}, x_{j},
       y)\, dx_{i}\, dx_{j}\,
       dy\right)^{2}\Biggr]\Biggr| \\
     & \leq \gamma\bigl(\Vert f\Vert_{\infty},
     \Vert{\eta}\Vert_{\infty}, \Delta \bigr)
     \left[\frac{\vert{M_{n}}\vert}{n} + \Vert{S_{M}f-f}\Vert_{2} +
       \E[{\Vert{S_{M}\hat{g}-\hat{g}}\Vert_{2}}]\right]
   \end{align*}
   Note that
   \begin{align*}
     \E[{\Vert{S_{M_{n}}\hat{g} - \hat{g}}\Vert_{2}}] \leq &
     \E[{\Vert{S_{M}\hat{g}-S_{M}g}\Vert_{2}}] +
     \E[{\Vert{\hat{g}-g}\Vert_{2}}]
     + \E[{\Vert{S_{M_{n}}g-g}\Vert_{2}}] \\
     & \leq 2\E[{\Vert{\hat{g}-g}\Vert_{2}}] +
     \E[{\Vert{S_{M_{n}}g-g}\Vert_{2}}]
   \end{align*}
   where $g(x_{i}, x_{j}, y) = \int H_{3}(f, x_{i}, x_{j}, x_{i2},
   x_{j2}, y) f(x_{i2}, x_{j2}, y) \, dx_{i2}\, dx_{j2}$. %

   The second term converges to 0 since $g\in\mathbb{L}^{2}(\, dx\,
   dy\, dz)$ and for all $t$ belonging to $\mathbb{L}^{2}(\, dx\, dy\,
   dz)$, $\int\left(S_{M}t-t\right)^{2} \, dx\, dy\, dz \rightarrow
   0$. %
   Moreover
   \begin{align*}
     & \Vert{\hat{g}-g}\Vert_{2}^{2} \\
     =\ & \int\left[ \hat{g}(x_{i}, x_{j}, y) - g(x_{i}, x_{j},
       y)\right]^{2}
     \, dx_{i}\, dx_{j}\, dy \\
     =\ & \int\Bigg[ \int\Big(H_{3}(\hat{f}, x_{i}, x_{j}, x_{i2},
     x_{j2}, y) \\
     & \qquad - H_{3}(f, x_{i}, x_{j}, x_{i2}, x_{j2}, y) \Big)
     f(x_{i2}, x_{j2}, y)
     \, dx_{i2}\, dx_{j2}\Bigg]^{2}\, dx_{i}\, dx_{j}\, dy \\
     \leq\ & \int\Bigg[ \int\Big(H_{3}(\hat{f}, x_{i}, x_{j}, x_{i2},
     x_{j2}, y) \\
     & \qquad - H_{3}(f, x_{i}, x_{j}, x_{i2}, x_{j2}, y)
     \Big)^{2} \, dx_{i2}\, dx_{j2}\Bigg] \\
     & \qquad \left[\int f(x_{i2}, x_{j2}, y)^{2} \, dx_{i2}\,
       dx_{j2}\right] \, dx_{i}\, dx_{j}\, dy \\
     \leq\ & \Delta \Vert f\Vert_{\infty}^{2} \int\Big(H_{2}(\hat{f},
     x_{i}, x_{j}, x_{i2}, x_{j2}, y) \\
     & \qquad - H_{2}(f, x_{i}, x_{j}, x_{i2}, x_{j2}, y) \Big)^{2}
     \, dx_{i}\, dx_{j}\, dx_{i2}\, dx_{j2}\, dy\\
     \leq\ & \delta\Delta^2\Vert f\Vert_{\infty}^{2}
     \int\left(\hat{f}(x_{i}, x_{j}, y) - f(x_{i}, x_{j},
       y)\right)^{2} \, dx_{i}\, dx_{j}\, dy
   \end{align*}
   for some constant $\delta$ that comes out of applying the mean
   value theorem to $H_{3}(\hat{f}, x_{i}, x_{j}, x_{i2}, x_{j2}, y) -
   H_{3}(f, x_{i}, x_{j}, x_{i2}, x_{j2}, y)$. %
   The constant $\delta$ was taken under Assumptions
   \ref{ass:taylor:A1}-\ref{ass:taylor:A3}. %
   Since $\E[{\Vert{f-\hat{f}}\Vert_{2}}]\to0$ then
   $\E[{\Vert{g-\hat{g}}\Vert_{2}}]\to0$. %
   Now show that the expectation of
   \begin{multline*}
     \int\hat{g}(x_{i}, x_{j}, y)^{2} f(x_{i}, x_{j}, y)\, dx_{i}\,
     dx_{j}\, dy \\
     - \left(\int\hat{g}(x_{i}, x_{j}, y) f(x_{i}, x_{j}, y) \,
       dx_{i}\, dx_{j}\, dy\right)^{2}
   \end{multline*}
   converges to 0. We develop the proof for only the first term. We
   get
   \begin{align*}
     & \Bigg\vert{\int\hat{g}(x_{i}, x_{j}, y)^{2}f(x_{i}, x_{j}, y)\,
       dx_{i}\, dx_{j}\, dy-\int g(x_{i}, x_{j}, y)^{2}f(x_{i}, x_{j},
       y)\, dx_{i}\, dx_{j}\, dy}\Bigg\vert \\
     & \leq \int\Big\vert{\hat{g}(x_{i}, x_{j}, y)^{2} - g(x_{i},
       x_{j}, y)^{2}} \Big \vert f(x_{i}, x_{j}, y)\,
     dx_{i}\, dx_{j}\, dy \\
     & \leq \lambda\int\left(\hat{g}(x_{i}, x_{j}, y) -
       g(x_{i}, x_{j}, y)\right)^{2} \, dx_{i}\, dx_{j}\, dy \\
     & = \lambda\Vert{\hat{g}-g}\Vert_{2}^{2}
   \end{align*}
   for some constant $\lambda$. %
   By taking the both sides expectation, we see it is enough to show
   that $\E[{\Vert{\hat{g}-g}\Vert_{2}^{2}}]\to0$. Besides, we can
   verify
   \begin{align*}
     g(x_{i}, x_{j}, y) =\ & \int H_{3}(f, x_{i}, x_{j}, x_{i2},
     x_{j2}, y)
     f(x_{i2}, x_{j2}, y) \, dx_{i2}\, dx_{j2} \\
     =\ & \frac{2}{\int f(x_{i}, x_{j}, y) \, dx_{i}\, dx_{j}}
     \left(x_{i}-\hat{m}_{i}(y)\right) \\
     & \left(\int x_{j2}f(x_{i2}, x_{j2}, y)\, dx_{i2}\, dx_{j2} -
       \hat{m}_{j}(y)\int f(x_{i2}, x_{j2}, y) \, dx_{i2}\,
       dx_{j2}\right) \\
     =\ & 0
   \end{align*}
   which proves that the expectation of $\int\hat{g}(x_{i}, x_{j},
   y)^{2} f(x_{i}, x_{j}, y)\, dx_{i}\, dx_{j}$ converges to 0. %
   Similar computations shows that the expectation of
   \begin{equation*}
     \left(\int\hat{g}(x_{i}, x_{j}, y) f(x_{i}, x_{j}, y) \, dx_{i}\,
       dx_{j}\right)^{2}
   \end{equation*}
   also converges to 0. %

   Finally we have
   \begin{equation*}
     \lim_{n\to\infty}n\, \E\bigl[\widehat{Q}-Q\bigr]^{2}=0. 
     \qedhere
   \end{equation*}
 \end{proof}


\end{document}